\documentclass[11pt,reqno,pdfoutput=1]{amsart}
\usepackage{amscd}
\usepackage{amssymb}
\usepackage{amsthm}
\usepackage{amsfonts}
\usepackage{amsmath}
\usepackage{appendix}

\usepackage{arydshln}
\usepackage{boldline}
\usepackage{booktabs}
\usepackage{comment}
\usepackage{calc}
\usepackage{color}
\usepackage{colortbl}
\usepackage{colonequals}
\usepackage{enumerate}
\usepackage{epstopdf}
\usepackage{MnSymbol}
\usepackage[margin=1.25in,letterpaper]{geometry}   
\usepackage{graphicx}

\usepackage{latexsym}
\usepackage{leftindex}
\usepackage{longtable}
\usepackage{makecell}
\usepackage{mathtools}
\usepackage{multirow}
\usepackage{nicematrix}
\usepackage{tensor}
\usepackage{tikz-cd}
\usepackage{xcolor}
\usepackage[all]{xy}
\usepackage{xypic}
\usepackage[utf8]{inputenc}
\usepackage[colorlinks]{hyperref}
\hypersetup{
 colorlinks=true,
 linkcolor=blue,
 filecolor=magenta, 
 urlcolor=cyan,
}
\usepackage[nameinlink, capitalize]{cleveref}
\usepackage{colonequals}
\allowdisplaybreaks[1]
\newtheorem{theorem}{Theorem}[section]
\newtheorem{thmy}{Theorem}
\newenvironment{thmx}{\stepcounter{theorem}\begin{thmy}}{\end{thmy}}

\newtheorem*{introtheorem*}{Main Theorem}

\newtheorem{corollary}[theorem]{Corollary}
\newtheorem{lemma}[theorem]{Lemma}
\newtheorem{proposition}[theorem]{Proposition}

\theoremstyle{definition}
\newtheorem{definition}[theorem]{Definition}
\newtheorem*{defn*}{Definition}
\newtheorem{chunk}[theorem]{}
\newtheorem{ex}[theorem]{Example}

\newtheorem{remark}[theorem]{Remark}

\newtheorem{notation}[theorem]{Notation}
\newtheorem{setting}[theorem]{Setting}

\theoremstyle{remark}

\numberwithin{equation}{theorem}
\newcommand{\BZ}{{\mathbb Z}}
\newcommand{\kk}{{\sf{k}}}
\newcommand{\m}{{\mathfrak m}}

\newcommand{\fh}{{\mathfrak h}}
\newcommand{\sym}{\mathfrak{S}}
\def\Hilb{\operatorname{Hilb}}

\def\ann{\operatorname{ann}}

\def\ch{\operatorname{char}}

\def\init{\operatorname{in}}

\def\rank{\operatorname{rank}}
\def\reg{\operatorname{reg}}

\def\Ann{{\operatorname{Ann}}}
\def\Hess{{\operatorname{Hess}}}
\def\Tor{{\operatorname{Tor}}}

\newcommand{\HH}{\operatorname{H}}

\def\Soc{\operatorname{Soc}}
\def\Supp{\operatorname{Supp}}

\def\ov{\overline}

\newcommand{\Pu}[1]{P_{\,[#1]}}
\title[]{A study of a quadratic almost complete intersection ideal and its linked Gorenstein ideal}


\author[Diethorn]{Rachel Diethorn}
\email{rdiethor@oberlin.edu}
\address{Department of Mathematics, Oberlin College, Oberlin, OH}
\author[G\"unt\"urk\"un]{Sema G\"unt\"urk\"un}
\email{s.gunturkun@essex.ac.uk}
\address{School of Mathematics, Statistics and Actuarial Science, University of Essex, UK}
\author[Hardesty]{Alexis Hardesty}
\email{ahardesty1@twu.edu}
\address{Division of Mathematics, Texas Woman's University, Denton, TX}
\author[Mete]{Pinar Mete}
\email{pinarm@balikesir.edu.tr}
\address{Department of Mathematics, Balikes\.{ı}r University, Turkey}
\author[\c{S}ega]{Liana \c{S}ega}
\email{segal@umkc.edu}
\address{Division of Computing, Analytics and Mathematics, University of Missouri-Kansas City, Kansas City, MO}
\author[Sobieska]{Aleksandra Sobieska}
\email{sobieskasnyd@marshall.edu}
\address{Department of Mathematics \& Physics, Marshall University, Huntington, WV}
\author[Veliche]{Oana Veliche}
\email{o.veliche@northeastern.edu}
\address{Department of Mathematics, Northeastern University, Boston, MA}

\begin{document}

%\tableofcontents
%%%%%%%%%%%%%%%%%%%%%%%%%%%%%%%%%%%%%%%%%%%%%%%%%%%%%%%%%%%%%%%%%%%%%%%%%%%%%%%%%%%%%%%%%%%%%%%%%%%%%%%%%%
\begin{abstract}
We examine the ideal $I=(x_1^2, \dots, x_n^2, (x_1+\dots+x_n)^2)$ in the polynomial ring $Q=\kk[x_1, \dots, x_n]$, where $\kk$ is a field of characteristic zero or greater than $n$. We also study the Gorenstein ideal $G$ linked to $I$ via the complete intersection ideal $(x_1^2, \dots, x_n^2)$. We compute the Betti numbers of $I$ and $G$ over $Q$ when $n$ is odd and extend known computations when $n$ is even. A consequence is that the socle of $Q/I$ is generated in a single degree (thus $Q/I$ is level) and its dimension is a Catalan number.  We also describe the generators and the initial ideal with respect to reverse lexicographic order for the Gorenstein ideal $G$. 
\end{abstract}
\maketitle
%%%%%%%%%%%%%%%%%%%%%%%%%%%%%%%%%%%%%%%%%%%%%%%%%%%%%%%%%%%%%%%%%%%%%%%%%%%%%%%%%%%%%%%%%%%%%%%%%%%%%%%%%%%%
\section{Introduction}
The ring $R=\kk[x_1, \dots, x_n]/I$ with $I=(x_1^2, \dots, x_n^2, (x_1+\dots+x_n)^2)$ and $\kk$ a field emerged recently as a protagonist of several papers (\cite{BV}, \cite{kling}) that are concerned with the failure of the Weak Lefschetz Property. This ring belongs to the class of almost complete intersection rings generated by $n+1$ general forms, studied by Migliore and Mir\'o-Roig in \cite{MM}. It satisfies many properties that hold under generic assumptions, due to the known fact that the complete intersection ring $\kk[x_1, \dots, x_n]/J$ with $J=(x_1^2, \dots, x_n^2)$ has the Strong Lefschetz Property when the characteristic of $\kk$ is zero (see \cite{Stanley80,Watanabe85, RRR}) or greater than $n$ (see \cite{cook}). 

For $n\geq 2$, set $Q=\kk[x_1, \dots, x_n]$ and $\ell=\left\lfloor\frac{n-2}{2}\right\rfloor$. Since $J$ is a complete intersection, the linked ideal $G=J:I$ defines a Gorenstein ring $A=Q/G$. When $n$ is odd, we describe the graded Betti numbers of $R$ and $A$ over $Q$ in \cref{c:RoverQ betti} and \cref{c:AoverQ betti}, respectively. For example, the Betti table for $R$ over $Q$ when $n\ge 7$ is shown below. 

\begin{thmx}[\cref{c:RoverQ betti}]\label{A}
Let $n\geq 7$ be an odd integer and $\kk$ a field with $\ch\kk>n$ or $\ch\kk =0$. Then the Betti table of $R$ over $Q$ is: 

\begin{center}
\setlength{\tabcolsep}{3pt} 
\begin{tabular}{r| c c c c c c c c c c c c c c }
& \small{$0$} & \small{$1$} & \small{$2$} & \small{$3$}  & $\ldots$ & \small{$\ell$} & \small{$\ell+1$} & \small{$\ell+2$} & \small{$\ell+3$} & \small{$\ell+4$} & $\ldots$ & \small{$n-1$} & \small{$n$} \\
\hline
$0$ & $\binom{n+1}{0}$ & - & - & - & $\cdots$  & - & - & - & 
- & - & $\cdots$ & - &-\\ 
$1$ & -  & $\binom{n+1}{1}$ & - & - &$\cdots$ & - & - & - & - & - & $\cdots$ & -  & - \\ 
\vphantom{\huge{A}}$\vdots$&$\vdots$&$\vdots$&$\vdots$&$\vdots$&$\ddots$&$\vdots$&$\vdots$&$\vdots$&$\vdots$ & $\vdots$&$\vdots$&$\vdots$&$\vdots$ \\ 
\vphantom{\huge{A}} $\ell$ &  - & - & - & - & - &$\binom{n+1}{\ell}$ &- &- & - & - & $\cdots$ & - & - \\
\vphantom{\huge{A}}$\ell+1$ & - & - & $\rho_{2}$ &$\rho_3$ &$\cdots$ & $\rho_{_{\ell}}$ & $\rho_{_{\ell+1}}+\binom{n}{
\ell}$ & $\rho_{_{\ell+2}}$ & $\rho_{_{\ell+3}}$ & $\rho_{_{\ell}}$ &$\ldots$ & $\rho_{2}$ & - \\
$\ell+2$& - & - & - & $\rho_2$ &$\cdots$ &$\rho_{_{\ell-1}}$& $\rho_{_{\ell}}$ & $\rho_{_{\ell+1}}$ & $\rho_{_{\ell+2}}$ & $\rho_{_{\ell+3}}$ & $\cdots$ &  $\rho_{3}$ &$\rho_{2}$ \\
\end{tabular}
\end{center}
where $\rho_i=\rho_i(n-1)$ is  defined recursively in \cref{not: rho} in terms of binomial coefficients and $\rho_2$ is equal to the Catalan number $C_{\ell+2}$. 
\end{thmx}

In the Betti table above, the entries marked ``-'' are, by convention, $0$. In particular, it follows that $R$ has socle generated in a single degree (in other words, $R$ is level) and the dimension of the socle of $R$ is equal to $C_{\ell+2}$. In fact, these statements hold for all $n\ge 2$, regardless of parity (see \cref{Cataln-socle-gens}).  

When $n$ is even, the Hilbert functions of the rings $R$ and $A$ and their Betti tables over $Q$ can be obtained as special cases of results of \cite{MM} on  almost complete intersection ideals generated by $n+1$ general forms. In \cref{betti RQ} and \cref{betti AQ} we revisit these results and extend them by recording symmetries  implied by  Matlis duality. To establish the Betti tables when $n$ is odd, we first argue in \cref{l:R-ezd} that $x_n$ is an exact zero divisor (as defined in \cite{HenriquesSega11}) for both rings $R$ and $A$, and leverage this to conclude that $R$ and $A$ are liftable to $Q$ in the sense of \cite{Dao} (see \cref{t:RAoverQ}).  Finally we use the Betti tables from the even case to complete the proofs of \cref{c:RoverQ betti} and \cref{c:AoverQ betti}.

A consequence of the Betti tables of $R$ and $A$ is that the ideal $G/J$ is generated in degree $\ell+1$ by $C_{\ell+2}$ elements (see \cref{Cataln-socle-gens}). In \cref{thm: G generators} we describe the ideal $G$ and its initial ideal with respect to reverse lexicographic order as follows, where $(f)_{\sym_n}$ denotes the principal symmetric ideal generated by $f\in Q$, which is the ideal generated by all polynomials obtained from $f$ by permuting the variables. 

\begin{thmx}[\cref{thm: G generators}]
\label{B}
If $\kk$ a field with $\ch\kk>n$ or $\ch\kk =0$, then the Gorenstein ideal  $G$ is generated as follows:
\begin{align*}
%\label{eq:gensG}
G = \begin{cases}
J+\left((x_1-x_2)\cdots (x_{n-2}-x_{n-1})\right)_{\sym_n}, &\text{if $n$ is odd}\\
J+\left((x_1-x_2)\cdots (x_{n-3}-x_{n-2})x_{n-1}\right)_{\sym_n}, &\text{if $n$ is even}\;.
\end{cases}
\end{align*}
and the initial ideal of $G$ is 
\[
\init_{>}(G)=J+\big(x_{i_1}x_{i_2}\dots x_{i_{\ell+1}}\,\vert\, 0<i_1<\dots<i_{\ell+1}, \, i_j\le 2j\text{ for all $j\in [\ell+1]$}\big)\,.
\]
In particular, $G$ has a Gr\"obner basis generated in the same degrees as $G$. 
\end{thmx}

Our work is motivated by the desire to understand ideals generated by $n+1$ general quadrics, in the spirit of \cite{MM}. We subsequently pursue this goal in \cite{aci}, where we study (infinite) minimal free resolutions over rings defined by such ideals. There, we parametrize the ideals by a projective space, and prove that properties of interest hold for ideals corresponding to a nonempty open set. The rings $R$ and $A$ in this paper are precisely the ones used in \cite{aci} to establish that the  open sets constructed there are nonempty, and thus our results are intended to lay the groundwork for a more general investigation. 

The paper is organized as follows. In \cref{sec:prelim} we collect some preliminary facts regarding Hilbert functions, Catalan numbers, and Lefschetz properties which we will use throughout the paper. In \cref{sec:consequences} we work with a more general setting towards establishing various consequences of Lefschetz properties needed in the proofs of our main results. In particular, we investigate a more general class of quadratic almost complete intersections, which contains the ring $R$, and their linked Gorenstein rings, and we compute their Betti numbers when $n$ is even. %The results of this section apply, more generally, to rings defined by $n+1$ general quadrics, which are studied at length in \cite{aci}. 
In \cref{sec: aci exp} we apply the results from \cref{sec:consequences} to the rings $R$ and $A$ and we compute their Betti numbers when $n$ is odd, proving \cref{A}. In \cref{G} we prove \cref{B} and we show that the Gorenstein ring $A$ satisfies the Strong Lefschetz Property when $\ch \kk=0$, using the Hessian criterion for the Macaulay inverse form of $G$.  

\section{Preliminaries}\label{sec:prelim}

In this section we record notation, terminology and some standard results concerning Hilbert functions, Catalan numbers and Lefschetz properties.  
Let $S$ be a standard graded Noetherian $\kk$-algebra with $\kk$ a field and $M$ a finitely generated graded $S$-module with graded components $M_n$.  

We denote by $\fh_M : \BZ \to \BZ_{\geq 0}$ the Hilbert function of $M$; that is,
\begin{align*}
\fh_M(n)&=\dim_{\kk}M_n\quad \text {for all $n\in \mathbb Z$}. %ola comma
\end{align*}
We also denote by $\Hilb_M(t)$ the Hilbert series of $M$; that is,
\begin{align*}
\Hilb_M(t)=\sum_{n= m}^{\infty}\fh_M(n)t^n, 
\end{align*}
where $m$ is the smallest integer with $\fh_M(m)\ne 0$.

For $i,j\in \mathbb Z$ we denote by $\beta_{i,j}^S(M)$ the $i$th graded Betti number of $M$ over $S$ in internal degree $j$, namely 
\[
\beta_{i,j}^S(M)=\rank_\kk(\Tor_i^S(M,\kk)_j)\,.
\]
The regularity of $M$ over $S$ is defined by 
\[
\reg_S(M)=\max\{j-i \,\vert\, \beta_{i,j}^S(M)\ne 0\}\,.
\]

The following is a standard computation of Betti numbers.

\begin{lemma}
\label{betti-strand}
Let $\kk$ be a field, $S$ a standard graded Noetherian $\kk$-algebra, 
and  $M$ a finitely generated graded $S$-module generated in nonnegative degrees.  
Then for all integers $a,k$ with $k\ge 0$ we have: 
\begin{align*}
\beta_{k,k+a}^S(M)=(-1)^k\fh_M(k+a)
-&\sum_{i=0}^{k-1}(-1)^{i+k}\fh_S(k-i)\beta_{i,i+a}^S(M)\\
-&\sum_{i=0}^{k+a} (-1)^{i+k}\sum_{\substack{j=i\\j\ne i+a}}^{k+a}
\fh_{S}(k+a-j)\beta_{i,j}^S(M)\,.
\end{align*}
\end{lemma}

\begin{proof}
Consider an augmented  minimal graded free resolution of $M$ over $S$
\[
\cdots\to F_{i+1}\to F_{i}\to \dots \to F_1\to F_0\to M\to 0,
\]
 where
\[
F_i=\bigoplus_{j\geq 0} S(-j)^{\beta_{i,j}^S(M)}\quad \text{for all}\quad  i\geq 0. %ola comma
\]
A dimension count in degree $k+a$ gives:
\begin{align*}
\fh_M(k+a)&=\sum_{i\geq 0}(-1)^i\fh_{F_i}(k+a)\\&=
\sum_{i\ge 0} (-1)^i\sum_{j\ge 0} \fh_{S}(k+a-j)\beta_{i,j}^S(M)\\
&=\sum_{i=0}^{k+a} (-1)^i\sum_{j=i}^{k+a} \fh_{S}(k+a-j)\beta_{i,j}^S(M)\\
&=\sum_{i=0}^{k+a} (-1)^i\Big(\fh_{S}(k-i)\beta_{i,i+a}^S(M)+\sum_{\substack{j=i\\j\ne i+a}}^{k+a}\fh_{S}(k+a-j)\beta_{i,j}^S(M)\Big)\\
&=  (-1)^{k}\fh_S(0)\beta_{k,k+a}^S(M)+\sum_{i=0}^{k-1} (-1)^i\fh_{S}(k-i)\beta_{i,i+a}^S(M) \\
&\hspace{4.5cm} + \sum_{i=0}^{k+a} (-1)^i\sum_{\substack{j=i\\j\ne i+a}}^{k+a} \fh_{S}(k+a-j)\beta_{i,j}^S(M).
\end{align*} 
Since $\fh_S(0)=1$, this gives the desired equation in the statement.
\end{proof}

Next we record the definition of the Catalan numbers which appear often throughout our calculations. This sequence of numbers has many combinatorial interpretations, many of which can be found in \cite{stanley-catalan} and at OEIS \cite{catalan-oeis}.

\begin{definition}[{\bf Catalan Numbers}]\label{Catalan}
The $k$th Catalan number is defined by  
\[
C_k = \frac{1}{k+1}\binom{2k}{k}, \quad \text{ for } k\geq 0.\]
\end{definition}

\begin{remark}
\label{Catalan-rmk}
We make use of the following formulas: 
\begin{align*}
C_k&={\binom{2k}{k} - \binom{2k}{k+1}}
={\binom{2k-1}{k-1} - \binom{2k-1}{k+1}}={\binom{2k-2}{k-1}-\binom{2k-2}{k+1}}
\end{align*}
where the first equality is a well-known presentation of $C_k$ and the second and third follow from Pascal's identity. 
\end{remark}

Finally, we recall the definitions of various Lefschetz properties we will use throughout the paper and observe that our main example of interest satisfies these properties.
  
\begin{definition}[{\bf Lefschetz Properties}] 
Let $f$ be a homogeneous element of $S$ of degree $d\geq 1$. We say that $f$ is a {\it maximal rank element} of $S$ if for each $i\ge 0$ the map  $S_{i}\xrightarrow{}S_{i+d}$  given by multiplication by $f$ has maximal  rank.  If $d=1$ (that is, if $f$ is a linear form) and $f$ is a maximal rank element, then we say that $f$ is a {\it Weak Lefschetz element}.  If there exists a linear element $f$ in $S$  such that for each $j\geq 1$  the power  $f^j$ is a maximal rank element of $S$, then we say that $f$ is a {\it Strong Lefschetz element} of $S$. 

We say that the ring $S$ has the {\it Weak (respectively Strong) Lefschetz Property, or WLP (respectively SLP)}, if there exists a Weak (respectively Strong) Lefschetz element of $S$. Notice that if $S$ has SLP with Strong Lefschetz element $f$, then it has WLP with Weak Lefschetz element $f$, and $f^j$ is a maximal rank element for each $j\ge 1$.
\end{definition}

\begin{remark}
\label{ex: wlp}
Let $Q=\kk[x_1, \dots, x_n]$ be a polynomial ring in $n$ variables, with $n\ge 2$. When the characteristic of $\kk$ is zero, any Artinian monomial complete intersection quotient of $Q$  has the SLP (see \cite{Stanley80, Watanabe85, RRR}) and when the characteristic is greater than $n$, any Artinian monomial complete intersection quotient of $Q$ with generators in the same degree has the SLP (see \cite{cook}). Thus by semicontinuity, these properties extend to \textit{general} complete intersections with generators in fixed degrees, with characteristic assumptions as above (see for example \cite{MMR03}).

In particular,  $x_1+\dots+x_n$ is a Strong Lefschetz element of $Q/(x_1^2, \dots, x_n^2)$, and hence $(x_1+\dots+x_n)^2$ is a maximal rank element of $Q/(x_1^2, \dots, x_n^2)$, when $\ch\kk=0$ by \cite[Proposition 2.2]{MMN11} or when $\ch \kk> n$ by \cite[Theorem 3.6(ii)]{cook}. 
\end{remark}

\section{Consequences of Lefschetz Properties }
\label{sec:consequences}
While the main focus of this paper is on the almost complete intersection ring $R$ and the Gorenstein ring $A$  defined in the introduction, in this section we work with a more general setting towards  establishing ingredients needed for our main results.  In view of \cref{ex: wlp}, the rings $R$ and $A$ in the introduction satisfy the settings and the assumptions of our results below.  Furthermore, under appropriate conditions on the characteristic of $\kk$, the results of this section apply, more generally, to rings defined by $n+1$ general quadrics, which are studied at length in \cite{aci}.  

Throughout this section we assume the following setting: 

\begin{setting}
\label{not:1}
Let $\kk$ be a field, $Q=\kk[x_1, \dots, x_n]$ with $n\geq 2$, and $\{f_1,\dots, f_{n+1}\}$ a set of homogeneous elements in $Q$ of degree $2$. Define the following ideals of $Q$:
\begin{equation*}
J \coloneq (f_1, \dots, f_n),\quad I\coloneq (f_1, \dots, f_{n+1}),\quad \text{and}\quad G \coloneq  J:I. 
\end{equation*}
We further assume that $f_1,\ldots, f_n$ form a regular sequence and $f_{n+1}\notin J$, so that $J$ is a complete intersection ideal, $I$ is an almost complete intersection ideal, and $G$ is a Gorenstein ideal by \cite[Remark 2.7]{HU87}. We also define the following quotient rings of $Q$:
\begin{equation*}
P \coloneq Q/J,\quad R\coloneq  Q/I,\quad\text{and}\quad A\coloneq Q/G. 
\end{equation*}
Notice that $\ann_P(f_{n+1})=G/J$ and $A\cong P/\ann_P(f_{n+1})$.

We also set 
\[
\ell\coloneq \left\lfloor\frac{n-2}{2}\right\rfloor.
\]
Observe that $n=2\ell+2$ when $n$ is even and $n=2\ell+3$ when $n$ is odd. 

In this setting, $P$ is a complete intersection ring, $R$ is an almost complete intersection ring, and   $A$ is a Gorenstein ring. 
\end{setting}

\begin{remark}
\label{n2}
When $n=2$ in \cref{not:1},  we have $\dim_\kk[(x_1,x_2)^2]_2=3=\dim_\kk [I]_2\,$,
thus 
\[I=(x_1,x_2)^2\quad\text{and}\quad R=\kk[x_1,x_2]/(x_1,x_2)^2.\] 
Furthermore, since $f_1, f_2$ is a regular sequence of quadrics, we have $P_3=0$, and hence $(x_1,x_2)^3\subseteq (f_1,f_2)=J$. Since $G=J:I$ and $I=(x_1,x_2)^2$, this implies $(x_1,x_2)\subseteq G$. Thus $G=(x_1,x_2)$ and $A=\kk$.  
\end{remark}

\begin{remark}\label{APR}
With notation and hypotheses in \cref{not:1}, there exists a short exact sequence of graded $P$-modules:
\begin{equation}
\label{ses APR}
0\to A(-2)\to P\to R\to 0; 
\end{equation}
see for example \cite{Eis}, where the map $P\to R$ is the canonical projection and the map $A(-2)\to P$ is multiplication by $f_{n+1}$.  As a direct consequence of this exact sequence we have the following equality: 
\begin{align}
\label{hf APR}
\fh_A(i) & = \fh_P(i+2)-\fh_R(i+2)\quad \text{for all }\,  i\geq 0.
\end{align}
\end{remark}

\begin{lemma} 
\label{bettiLR}
Adopt the notation and hypotheses in \cref{not:1}. 
Then we have:
\begin{align*}
\beta_{i,j}^Q(G/J)= \beta_{n-i,2n-j}^Q(R)\quad \text{for all}\quad i,j\geq 0.
\end{align*}
\end{lemma}

\begin{proof}
First we claim that we have the isomorphism 
\begin{align*}
    G/J\cong \omega_R(-n),
\end{align*}
where $\omega_R$ is the canonical module of $R$. 
 Indeed, by \cite[Lemma 2]{Nagel91}, since $I$ is linked to $G=J:I$ by the complete intersection $J$ and since the regularity index of $P$ is $n+1$, we have the short exact sequence 
$$  0 \to \omega_{R}(-n) \to P \to A\to 0.$$
This implies that  $\omega_{R}(-n) \cong (J:I) / J \cong G/J$, as claimed.

Thus we have $(G/J)^\vee\cong R(n)$, where $(G/J)^\vee$ denotes the Matlis dual of $G/J$ over $Q$.  
In view of \cite[Proposition 2.2 (v)]{Boij99}, this implies
\begin{align*}
\beta_{i,j}^Q(G/J)&=\beta_{n-i,n-j}^Q((G/J)^\vee)=\beta_{n-i,n-j}^Q(R(n))=\beta_{n-i,2n-j}^Q(R).\qedhere
\end{align*}
\end{proof}

In the next result, parts (1) and (2) are known computations of the Hilbert functions of $A$ and $R$, as can be found in \cite{MM}. We record these results and their proofs for completeness, since we will use them in subsequent results. We supplement these results with computations of the socle degrees of $R$ and $A$ and the multiplicity of $R$.  Further we show that the dimensions of the socle of $R$ in the highest degree and the ideal $G/J$ in the lowest degree are the same and we compute their common value. 

\begin{proposition} 
\label{Hilb R and A}
Adopt the notation and hypotheses in \cref{not:1}. If $f_{n+1}$ is a quadratic maximal rank element of $P$, then the following assertions hold.
\begin{enumerate}[\quad$(1)$]
\item The Hilbert function of $R$ is given by  
\begin{equation}\label{hilbR}
\fh_{R}(i)= \max {\left\{ \binom{n}{i}-\binom{n}{i-2},0\right\}} \quad\text{for all}\quad i\geq 0.
\end{equation}
In particular, we have:  
\begin{enumerate}[\quad$(a)$]
\item The socle of  $R$  has  maximum degree  $n-\ell-1$.
\item The multiplicity of $R$ is
\[
e(R)={\binom{n+1}{\ell+2}}.
\]
\end{enumerate}
\item The Hilbert function of $A$ is given by
\begin{equation}\label{hilbertA}
\fh_{A}(i)=\min\left\{ \binom{n}{i},\binom{n}{i+2}\right\} \quad\text{for all}\quad i\geq 0.
\end{equation}
In particular, we have:
\begin{enumerate}[\quad$(a)$]
    \item The socle degree of  $A$ is $n-2$.
    \item The $P$-ideal $G/J$ is generated in degrees at least $\ell+1$.
\end{enumerate}
\item The following equalities hold: 
\[
\dim_\kk(\Soc R)_{n-\ell-1}= {\binom{n}{\ell+1}-\binom{n}{\ell+3}} = C_{\ell+2}=\dim_\kk(G/J)_{\ell+1}.
\]
where $C_{\ell+2}$ is the $(\ell+2)$-th Catalan number. 
\end{enumerate}
\end{proposition}

\begin{proof}
(1): Since $P$ is a complete intersection defined by $n$ quadrics, its Hilbert series is given by
\begin{equation}\label{hilbP}
    \Hilb_P(u)=\frac{(1-u^2)^n}{(1-u)^n}=(1+u)^n.
\end{equation}
Thus, its Hilbert function is $\fh_P(i)=\binom{n}{i}$ for all $i\geq 0$. 
By our assumption on $f_{n+1}$, the multiplication map  $P_{i-2}\xrightarrow{\cdot f_{n+1}} P_{i}$ has maximal rank for all $i\geq 2$.

Note that $\rank_{\kk}P_{i}=\binom{n}{i}$ for all $i$. If 
$\binom{n}{i-2}\leq \binom{n}{i}$, 
then the multiplication map by $f_{n+1}$ is injective. If $\binom{n}{i-2}> \binom{n}{i}$,
then the multiplication map by $f_{n+1}$ is surjective.  
%and $R_i=0$.
Since $R_i=P_i/f_{n+1}P_{i-2}$ for all $i$, we have that $R_i$ has rank $\binom{n}{i}-\binom{n}{i-2}$ in the injective case, and $R_i=0$ in the surjective case.

(1$a$): It suffices to check that $\fh_R(n-\ell-1)>0$ and $\fh_R(n-\ell)$=0. This follows directly from (1) when considering the even and odd cases separately. 

(1$b$):  We use (1) and (1$a$) to obtain the equalities: 
\begin{align*}
e(R)&=\sum_{i\ge 0} \fh_R(i)=\sum_{i\ge 0}\max\left\{\binom{n}{i}-\binom{n}{i-2},0\right\}\\
&=\sum_{i=0}^{n-\ell-1} \binom{n}{i}-\binom{n}{i-2}\\
&=\binom{n}{0}+\binom{n}{1}+\left[\binom{n}{2}-\binom{n}{0}\right]+\left[\binom{n}{3}-\binom{n}{1}\right]+\dots+\left[\binom{n}{n-\ell-1}-\binom{n}{n-\ell-3}\right]\\
&=\binom{n}{n-\ell-1}+\binom{n}{n-\ell-2}=\binom{n+1}{n-\ell-1}=\binom{n+1}{\ell+2}.
\end{align*}

\noindent (2): By \eqref{hf APR} in \cref{APR} and part (1), we get:
\begin{align*}
 \fh_A(i) & = \fh_P(i+2)-\fh_R(i+2)\\
 &=\fh_P(i+2)-\max {\left\{ \binom{n}{i+2}-\binom{n}{i},0\right\}}\\
        & =\fh_P(i+2)-\max\{\fh_P(i+2)-\fh_P(i), 0\}\\
        & =\min\{\fh_P(i),\fh_P(i+2)\}\\
        & =\min\left\{ \binom{n}{i},\binom{n}{i+2}\right\}.
\end{align*}

(2$a$): This follows directly from the Hilbert function.

(2$b$): Since $A\cong P/(G/J)$, we need to show that $\fh_A(i)=\fh_P(i)$ for all $i<\ell+1$. Indeed, by (2),  the equality is equivalent to $\binom{n}{i}\le\binom{n}{i+2}$, which holds when $i<\ell+1$.

\noindent (3): By (2$b$) we have that $\beta_{0,\ell+1}^Q(G/J)$ is equal to the dimension of $(G/J)_{\ell+1}$. Using the fact that $R_{\ell+3}=P_{\ell+3}/f_{n+1}P_{\ell+1}$ is zero, we have the exact sequence
\[
0\to (G/J)_{\ell+1}\to P_{\ell+1}\xrightarrow{\cdot f_{n+1}} P_{\ell+3}\to 0.
\]
Now we can compute the dimension of $(G/J)_{\ell+1}$ as follows
\[
\dim_\kk (G/J)_{\ell+1}=\dim_\kk(P_{\ell+1})-\dim_\kk(P_{\ell+3})={\binom{n}{\ell+1}-\binom{n}{\ell+3}}.
\]
We then compute the dimension of $(\Soc(R))_{n-\ell-1}$ using \cref{bettiLR}: 
\[
\beta_{n,2n-\ell-1}^Q(R)=\beta_{0,\ell+1}^Q(G/J)={\binom{n}{\ell+1}-\binom{n}{\ell+3}}. 
\]
 By \cref{Catalan-rmk}, this expression is also equal to the Catalan number $C_{\ell+2}$.
\end{proof}

\begin{remark}
\label{maxrank-hilb}
Assume \cref{not:1}. Then the following are equivalent: 
\begin{enumerate}
    \item $f_{n+1}$ is a maximal rank element of $P$;
    \item The Hilbert function of $R$ is given by \eqref{hilbR};
    \item The Hilbert function of $A$ is given by \eqref{hilbertA}.
\end{enumerate}
Indeed, the equivalence (2)$\iff$(3) is a consequence of \cref{APR}, the implication (1)$\Rightarrow$(2) comes from \cref{Hilb R and A}(1), while the implication (2)$\Rightarrow$(1) can be easily proved by reversing the argument given in the proof of \cref{Hilb R and A}(1). 
\end{remark}

\begin{remark}
In view of the computation of the Hilbert series of $A$ in \cref{Hilb R and A}(2), the ring $A$ is $J$-compressed, in the terminology of \cite[Definition 2.3]{MRRG}.   Note that \cite[Theorem D]{MRRG} gives results about the socle of the ring $R$ when $\ch(\kk)=2$ and $J=(x_1^2, \dots, x_n^2)$. 
\end{remark}

In the next proposition we provide information on the Betti numbers of $A$ and $R$ over $Q$, and moreover, over the intermediate complete intersections defined therein.  

\begin{proposition}
\label{l: A and R over Pu}   Adopt the notation and hypotheses in \cref{not:1} and set 
\begin{equation*}
 \Pu{u} \coloneq Q/(f_1,\dots,f_{u}) \quad\text{for all}\quad 0\leq u\leq n %ola comma
\end{equation*}
with the convention that $\Pu{u}=Q$ when $u=0$. If $f_{n+1}$ is a quadratic maximal rank element of $P$, then for each $0\leq u\leq n$  the following equalities hold:

\begin{enumerate}[\quad $(1)$]
\item $\beta_{i,j}^{\Pu{u}}(A)= 
\begin{cases} 
\binom{n-u}{i},&\text{ if } j=2i\,  \text{ and }\,  0\leq j-i <\ell\\       
\quad 0,&\text{ if $j\ne 2i$ and $0\leq j-i<\ell$}\,;
\end{cases}$
\item 
$\beta_{i,j}^{\Pu{u}}(R)=
\begin{cases} 
\binom{n-u+1}{i},&\text{ if } j=2i\,  \text{ and }\,  0\leq j-i \le \ell\\       
\quad 0,&\text{ if $j\ne 2i$ and  $0\leq j-i\le \ell$}\,;
\end{cases}$ 
\end{enumerate}
Furthermore, the following equalities hold:
\begin{enumerate}[\quad $(3)$]
% \item  \label{betti APR} 
%  $\beta_{i,j}^P(R)=\beta_{i-1,j-2}^P(A)\quad \text{for all }\,  i\geq 1  \text{ and }  j\geq 0$, \\
\item[$(3)$] $\beta_{i,j}^Q(R)=0$ when $j-i>n-\ell-1$ and $\beta_{i,j}^Q(A)=0$ when $j-i>n-2$;
\item[$(4)$] For all integers $i,j$ with $ j\notin \{2i,2i-2\}$ there are equalities  
\label{betti RA}
\begin{equation*}
%\label{e:betti-switch}
\beta_{i,j}^Q(R)=\beta_{n-i+2,2n-j+2}^Q(R)=\beta_{i-1,j-2}^Q(A)=\beta_{n-i+1, 2n-j}^Q(A);
\end{equation*}
\item[$(5)$]  If $n$ is even, then 
\[
\beta_{\ell+1, 2\ell+2}^Q(R)=\beta_{\ell+3,2\ell+4}^Q(R)+\binom{n+1}{\ell+1}\,;
\] 
\item[$(6)$]  
$\beta_{n, 2n-\ell-1}^Q(R)=C_{\ell+2}$,
\[
\qquad
\beta_{2,\ell+3}^Q(R) =
\begin{cases}
C_{\ell+2},&\text{if } n\notin \{4,5\}\\
{\binom{n}{2}}+n+5, &\text{if } n\in \{4,5\}   
\end{cases}
\quad\text{and}\quad
\beta_{1, \ell+1}^Q(A) = 
\begin{cases}
C_{\ell+2},&\text{if } n\notin\{4,5\}\\
n+5,&\text{if } n\in \{4,5\}\,.
\end{cases}
\] 
\end{enumerate}
\end{proposition}

\begin{proof} (1): Set $L:=\ann_P(f_{n+1})=G/J$. Then, there exists an exact sequence: 
\begin{equation*}
    0\to L\to P\to A\to 0.
\end{equation*}
By \cref{Hilb R and A}(2$b$),  $L$ is generated in degrees at least $\ell+1$, and hence $\Tor_i^{\Pu{u}}(L,\kk)_j=0$ when $0\le j-i\le \ell$. From the long exact sequence of Tor modules:
\begin{align*}\label{Tor_les}
\cdots\to\Tor_i^{\Pu{u}}(L,\kk)_j\to \Tor_i^{\Pu{u}}(P,\kk)_j\to \Tor_i^{\Pu{u}}(A,\kk)_j\to \Tor_{i-1}^{\Pu{u}}(L,\kk)_j\to \cdots
\end{align*}
we obtain $\beta^{\Pu{u}}_{i,j}(A)=\beta_{i,j}^{\Pu{u}}(P)$ for all $0\le j-i<\ell.$ 
Since $P$ is a complete intersection over $\Pu{u}$ that is  generated by  $n-u$ quadrics we have  
\begin{equation}
\label{e:P-Pu}
\beta_{i,j}^{\Pu{u}}(P)
=\begin{cases} 
\binom{n-u}{i},&\ \text{if}\   j = 2i\\
0,&\ \text{otherwise}\,.  
\end{cases}
\end{equation}
The desired conclusion now follows.

(2): By \eqref{ses APR} in \cref{APR} there is a short exact sequence of graded 
$Q$-modules and in particular of $\Pu{u}$-modules:
$0\to A(-2)\to P\to R\to 0$, which induces a long exact sequence  of Tor modules:
\begin{equation}
\label{les}
    \xymatrix@C=1.8pc@R=1.2pc{ 
        & \Tor_{i}^{\Pu{u}}(A,\kk)_{j-2} \ar[rr]&&
          \Tor_{i}^{\Pu{u}}(P,\kk)_{j}   \ar[rr]&&
          \Tor_{i}^{\Pu{u}}(R,\kk)_{j}   \ar`r[r] `[l] `[llllld]`[dl] [dllll]&\\
        & \Tor_{i-1}^{\Pu{u}}(A,\kk)_{j-2} \ar[rr]&&
          \Tor_{i-1}^{\Pu{u}}(P,\kk)_{j}   \ar[rr]&&
          \Tor_{i-1}^{\Pu{u}}(R,\kk)_{j}
          }
\end{equation}

If  $j=2i$  and $0\leq i\le \ell$, then using part (1) and \eqref{e:P-Pu}, we get  $\Tor_i^{\Pu{u}}(A,\kk)_{2i-2}=0=\Tor_{i-1}^{\Pu{u}}(P,\kk)_{2i}$, so we have an exact sequence
\[
0\to \Tor_{i}^{\Pu{u}}(P,\kk)_{2i}\to 
\Tor_{i}^{\Pu{u}}(R,\kk)_{2i}\to \Tor_{i-1}^{\Pu{u}}(A,\kk)_{2i-2}\to 0.
\]
Counting ranks and using \eqref{e:P-Pu} and part (1) we get:
\begin{align*}
\beta_{i,2i}^{\Pu{u}}(R)&=\beta_{i,2i}^{\Pu{u}}(P)+\beta_{i-1,2i-2}^{\Pu{u}}(A)
=\binom{n-u}{i}+\binom{n-u}{i-1}
=\binom{n-u+1}{i}.
\end{align*}

If  $j\not =2i$  and $0\le j-i\le \ell$, then by using \eqref{e:P-Pu} and part (1), we get  $\Tor_{i}^{\Pu{u}}(P,\kk)_{j}=0=\Tor_{i-1}^{\Pu{u}}(A,\kk)_{j-2}$.
Hence, by \eqref{les} we obtain $\beta_{ij}^{\Pu{u}}(R)=0.$
  
(3):  This follows from the fact that since $R$ and $A$ are Artinian, their regularities are given by their socle degrees in \cref{Hilb R and A} parts (1a) and (2a).

(4):  Working over $Q$ and assuming $j\notin \{2i,2i-2\}$, the exact sequence \eqref{les} becomes
\[
0=\Tor_{i}^Q(P,\kk)_{j}\to \Tor_{i}^Q(R,\kk)_{j}\to \Tor_{i-1}^Q(A,\kk)_{j-2}\to \Tor_{i-1}^Q(P,\kk)_{j}=0
\]
where the computation of $\Tor^Q(P,\kk)$ on either end of the sequence follows from \eqref{e:P-Pu}.
%follows from the fact that a minimal free resolution of $P$ over $Q$ is given by the Koszul complex on the quadrics $f_1, \dots, f_n$ and so
% \begin{align}\label{eq:b2a}
%     \Tor^Q_a(P,\kk)_b=0\quad \text{when}\quad b\ne 2a.
% \end{align}
This yields the following equality of Betti numbers
\begin{align}\label{e:bettiRA}
    \beta_{i,j}^Q(R)=\beta_{i-1,j-2}^Q(A)
\end{align}
for all $i$ and $j$ with $ j\notin \{2i,2i-2\}$. 
 
Notice that 
%the integers $n-i+2$ and $2n-j+2$ satisfy the property that
$2n-j+2\notin\{2(n-i+2), 2(n-i+2)-2\}$, and thus by \eqref{e:bettiRA} we have 
\[
\beta_{n-i+2,2n-j+2}^Q(R)=\beta_{n-i+1,2n-j}^Q(A).
\]
Finally, since $A$ is Gorenstein, we observe that $\beta_{i-1,j-2}^Q(A)=\beta_{n-i+1,2n-j}^Q(A)$. Combining the three equalities of Betti numbers above finishes the proof.

(5): Writing \eqref{les} with $i=\ell+1$, $j=2i=2\ell+2$, and also with $i=\ell+3$, $j=2\ell+4$, we have exact sequences
%\small
\[
0=\Tor_{_{\ell+1}}^Q(A,\kk)_{_{2\ell}}\to \Tor_{_{\ell+1}}^Q(P,\kk)_{_{2\ell+2}}\to \Tor_{_{\ell+1}}^Q(R,\kk)_{_{2\ell+2}}\to \Tor_{_{\ell}}^Q(A,\kk)_{_{2\ell}}\to \Tor_{_{\ell}}^Q(P,\kk)_{_{2\ell+2}}=0
\]
\[
0\!=\!\Tor_{_{\ell+3}}^Q(P,\kk)_{_{2\ell+4}}\!\to\!\Tor_{_{\ell+3}}^Q(R,\kk)_{_{2\ell+4}}\!\to\! \Tor_{_{\ell+2}}^Q(A,\kk)_{_{2\ell+2}}\!\to\! \Tor_{_{\ell+2}}^Q(P,\kk)_{_{2\ell+4}}\!\to\!\Tor_{_{\ell+2}}^Q(R,\kk)_{_{2\ell+4}}\!=\!0
\]
where the terms equal to zero are due to \eqref{e:P-Pu} and the previously proved items (1) and (3). Computing ranks in these two sequences, we have: 
\begin{align}
\label{b1}\beta_{\ell+1,2\ell+2}^Q(R)&=\beta^Q_{\ell, 2\ell}(A)+\beta_{\ell+1, 2\ell+2}^Q(P)\\
\label{b2}\beta_{\ell+3,2\ell+4}^Q(R)&=\beta^Q_{\ell+2, 2\ell+2}(A)-\beta_{\ell+2, 2\ell+4}^Q(P)\,.
\end{align}
When $n$ is even, and hence $n=2\ell+2$, we use the fact that $A$ is Gorenstein to observe
\begin{align}
\label{b3}
\beta_{\ell, 2\ell}^Q(A)=\beta_{n-\ell, 2n-2-2\ell}^Q(A)=\beta_{\ell+2,2\ell+2}^Q(A)\,.
\end{align}
Combining \eqref{b3} with \eqref{b1}, \eqref{b2} and using \eqref{e:P-Pu}, we have
\begin{align*}
\beta_{\ell+1,2\ell+2}^Q(R)&=\beta_{\ell+3,2\ell+4}^Q(R)+\binom{n}{\ell+1}+\binom{n}{\ell+2}=\beta_{\ell+3,2\ell+4}^Q(R)+\binom{n+1}{\ell+2}\,.
\end{align*} 

(6): The equality $\beta_{n,2n-\ell-1}^Q(R)=C_{\ell+2}$ follows directly from \cref{Hilb R and A}(3). To prove the remaining equalities, assume first that $n\notin \{4,5\}$ and thus $\ell\neq 1$. Since $A=P/(G/J)$ and $\ell+1>2$, the Betti number $\beta_{1,\ell+1}^Q(A)$ is equal to the minimal number of generators of $G/J$ in degree $\ell+1$, which is equal to $C_{\ell+2}$ by \cref{Hilb R and A}(3). Notice that $\ell+3\notin\{2,4\}$, and thus by part (4) we have 
\begin{align*}
\beta_{2,\ell+3}^Q(R)=\beta_{1,\ell+1}^Q(A)=C_{\ell+2}.
\end{align*}
%$\beta_{2,\ell+3}^Q(R)=C_{\ell+2}$ as well.

Assume now $n\in \{4,5\}$, and thus $\ell=1$.  By \cref{Hilb R and A}(3), the minimal number of generators of $G/J$ in degree $\ell+1=2$ is $C_{\ell+2}=5$ and there are no generators in degree less than $2$. We conclude that the minimal number of generators of $G$ in degree $2$ is $n+5$ and hence $\beta_{1,2}^Q(A)=n+5$. 

Using \eqref{b1} with $\ell=1$, we have
\begin{align*}
\beta_{2,4}^Q(R)&=\beta_{2,4}^Q(P)+\beta_{1,2}^Q(A)=\binom{n}{2}+\beta_{1,2}^Q(A)=\binom{n}{2}+n+5.\qedhere
\end{align*}
\end{proof}

We finish this section by providing a complete description of the Betti numbers of $R$ and $A$ over $Q$ when $n$ is even. For the Betti numbers of $R$ over $Q$ we will need the following recursively defined sequence.

\begin{notation}
\label{not: rho} For $n\geq 1$, define a sequence $\{\rho_k(n)\}_{k\ge 0}$ by setting $\rho_0(n)=0$ and for $k\geq 1$ using  the following recursive relation: 
\begin{align*}
\rho_k(n)=&\sum_{i=0}^{k-1}(-1)^{i+k+1}{\binom{n+k-i-1}{n-1}}\rho_i(n)+\sum_{i=0}^{\ell} (-1)^{i+k+1}  \binom{n+k+\ell-2i}{n-1}\binom{n+1}{i}.
\end{align*}
Here are the sequences $\{\rho_k(n)\}_{0\leq k\leq n}$ for small even values of $n$:
\begin{align*}
    n=2:\quad & \{0,3,2\} \qquad &&  n=6:\quad \{0,0, 14,105,132,70,14\}\\
    %n=3:\qquad & \{0,6,8,3\}\\
    n=4:\quad & \{0,0,15,16,5\}\qquad 
    % n=5:\qquad & \{0,5,45,66,40,9\}\\
    && n= 8:\quad \{0,0, 42,288,945,1216,819,288,42\}.
\end{align*}
\end{notation}
The next result describes the Betti table of $R$ over $Q$ when the number of variables $n$ is even. Note that, in this table, the entries marked ``-'' are assumed to be zero. A different description of the Betti numbers can be found in \cite[Theorem 5.4]{MM}. \cref{betti RQ} also points out some properties of the sequence $\{\rho_k(n)\}_{0\leq k\leq n}$ that are not immediately clear from its definition.

\begin{proposition}[\bf Betti numbers of $R$ over $Q$] 
\label{betti RQ}
 Adopt the notation and hypotheses in \cref{not:1} and \cref{not: rho}, and assume $n$ is even.  
If $f_{n+1}$ is a quadratic maximal rank element of $P$, then
\[
\beta_{i,j}^Q(R)= \begin{cases} 
    \binom{n+1}{i},&\text{ if } j=2i\,  \text{ and }\,  0\leq i \leq\ell\\
       \rho_i(n), &\text{ if } j= i+\ell+1 \text{ and } i\ge 0\\
       0,&\text{ otherwise}\,.
    \end{cases}
\]
Consequently,  the sequence $\{\rho_i(n)\}_{i\ge 0}$ satisfies the following properties: 
\begin{enumerate}[\quad\rm(1)]
\item $\rho_1(n)=0$ when  $n\ge 4$ and $\rho_k(n)=0$ when $k>n$; 
  \item $\rho_k(n)=\rho_{n-k+2}(n)$ for all $0\leq k \leq \ell$; 
  \item $\rho_n(n)=C_{\ell+2}$ for all $n\ge 2$,  and $\rho_2(n)=C_{\ell+2}$ for all $n\ge 6$;
  \item $\rho_{\ell+1}(n)=\rho_{\ell+3}(n)+\binom{n+1}{\ell+1}$.
\end{enumerate}

In particular, when $n\ge 6$ and setting $\rho_k=\rho_k(n)$ for each $0\le k\le n$,  the Betti table of $R$ over $Q$ is:

\begin{center}
\begin{tabular}{r| c c c c c c c c c c c c c c}
& $0$ & $1$ & $2$ & $\ldots$  & $\ell$ & $\ell+1$ & $\ell+2$ & $\ell+3$ & $\ell+4$ & $\ldots$  & $n$ \\
\hline
$0$ & $\binom{n+1}{0}$ & - & -  & $\cdots$ & - & -& - & - & - &$\cdots$ & - \\ 
$1$ & -  & $\binom{n+1}{1}$ & - & $\cdots$ & - & - & - &- & -& $\cdots$ &-\\ 
%$2$ & - & -  &$\binom{n+1}{2}$& $\cdots$ & - & - & -& - & - & $\cdots$ &- \\ 
\vphantom{\huge{A}}\vdots&\vdots&\vdots&\vdots&$\ddots$&\vdots&\vdots&\vdots&\vdots&\vdots&\vdots&\vdots\\
$\ell$ &  - & - & - &$\cdots$&$\binom{n+1}{\ell}$ & -& -& - &- &$\cdots$ &-\\
$\ell+1$& - & - & $\rho_{2}$ &$\ldots$  & $\rho_{_{\ell}}$ & $\rho_{_{\ell+1}}$ & $\rho_{_{\ell+2}}$& $\rho_{_{\ell+3}}$& $\rho_{_{\ell}}$ &$\ldots$ & $\rho_{2}$ \\
\end{tabular}
\end{center}
\end{proposition}

\begin{proof} 
The hypothesis that $n$ is even implies $n=2\ell+2$. Since $\Pu{0}=Q$, by \cref{l: A and R over Pu}(2) and (3)  for $u=0$, we obtain:
\[
\beta_{i,j}^{Q}(R)=
\begin{cases} 
\binom{n+1}{i},&\text{ if } j=2i\,  \text{ and }\,  0\leq i \le \ell\\       
\quad 0,&\text{ if $j\ne 2i$ and  $j-i\le \ell$, or if}
\text{ $j-i\geq \ell+2$}\,.
\end{cases}
\]

It remains to determine $\beta_{k,k+\ell+1}^Q(R)$ for all $k\ge 0$. 

By \cref{Hilb R and A}(1)  we have:
$\fh_R(k+\ell+1)=\max\left\{ \binom{n}{k+\ell+1}-\binom{n}{k+\ell-1},0\right\}=0$  for all $k\geq 1$. We also have the equalities:
\begin{equation*}
\fh_Q(k-i)=\binom{n+k-i-1}{n-1}\quad\text{and}\quad
    \fh_Q(k+\ell+1-2i)=\binom{n+k+\ell-2i}{n-1}\quad\text{for all }k,i\geq 0.
\end{equation*}
Now, substituting  $M=R$, $S=Q$, $a=\ell+1$
%, and above equalities 
in \cref{betti-strand}, we obtain for $k\ge 1$:
\begin{align*}
\beta_{k,k+\ell+1}^Q(R)=(-1)^k\fh_R(k+\ell+1)-&\sum_{i=0}^{\ell}(-1)^{i+k}\fh_Q(k-i)\beta_{i,i+\ell+1}^Q(R)\\
-&\sum_{i=0}^{k+\ell+1} (-1)^{i+k}\sum_{\substack{j=i\\j\ne i+\ell+1}}^{k+\ell+1}
\fh_{Q}(k+\ell+1-j)\beta_{i,j}^Q(R).
\end{align*}
Plugging in the values of the Hilbert functions above and comparing with the definition of $\rho$, we see that $\beta_{k,k+\ell+1}^Q(R)=\rho_k$ for all $k\ge 1$. 

To show (1), note that $\rho_1=\beta_{1,\ell+2}^Q(R)=\beta_{0,\ell+2}^Q(I)$ and it is equal to $0$ when $\ell>0$ because the ideal $I$ is quadratic.     
The remaining conclusions about the numbers $\rho_i$ follow from from \cref{l: A and R over Pu} and the fact that since $n$ is even, we have $n-\ell+2=\ell+4$.  
\end{proof}

\begin{ex} \label{rem:n2n4bettiR}
The Betti tables for $R$ over $Q$ for $n=2$ and $n=4$, respectively, in \cref{betti RQ} are as follows:
\begin{center}
\begin{tabular}{r| c c c c }
& $0$ & $1$ & $2$   \\
\hline
$0$ & $1$ & -  & -\\ 
$1$ & - & $3$ & $2$  
\end{tabular}\qquad \qquad 
\begin{tabular}{r| c c c c c }
& $0$ & $1$ & $2$ & $3$  & $4$ \\
\hline
$0$ & $1$ & - & - & -& -\\ 
$1$ & -  & $5$ & - & - & - \\ 
$2$ & - & -  &$15$&$16$ & $5$  \\ 
\end{tabular}
\end{center}
\end{ex}

\begin{notation}
\label{not: gamma}
For $n\geq 1$, let $\{\gamma_k(n)\}_{k\geq 0}$ be a sequence such that 
\[
\gamma_0(n)=\begin{cases}
0, &\text{if $n>2$}\\
1, &\text{if $n=2$}\,, 
\end{cases} 
\quad 
\gamma_k(n) = \gamma_{n-k}(n) \text{ for } 1\le k\le n,
\] 
and for $1\le k \le \ell+1$,  we define $\gamma_k(n)$ recursively as follows
\[
\gamma_k(n) = 
(-1)^k{ \binom{n}{k+\ell+2}} - \sum_{i=0}^{k-1}(-1)^{i+k}{ \binom{n-1+k-i}{n-1}}\gamma_i(n) - \sum_{i=0}^{\ell-1}(-1)^{i+k}{ \binom{n-1+\ell+k-2i}{n-1}\binom{n}{i}}.
\]
Here are the sequences $\{\gamma_k(n)\}_{0\leq k\leq \ell+1}$  for small even values of $n$:
\begin{align*}
    n=2: & \quad\{1,2\} \qquad &&n=6:   \quad\{0,14, 85, 132\}\\ 
    % $n=3$:& \{0,-1\}\\
n=4: &  \quad\{0,9,16\}\qquad
    &&n=8:  \quad \{0,42,288,875,1216\}.
\end{align*}
\end{notation}

\begin{proposition}[\bf Betti numbers of $A$ over $Q$] 
\label{betti AQ}
Adopt the notation and hypotheses in \cref{not:1} and \cref{not: gamma} and assume $n$ is even. If $f_{n+1}$ is a quadratic maximal rank element of $P$, then
\[
\beta_{i,j}^Q(A)=\begin{cases} 
\binom{n}{i},&\text{ if } j = 2i   \text{ and } 0\leq  i \leq \ell-1 \text{ or  if }\ j = 2i-2   \text{ and } \ell+3\leq i \leq n\\
\gamma_i(n),&\text{ if}\ j=i+\ell \text{ and }\,  0\leq i\leq n\\
0,&\text{ otherwise}\,.
\end{cases}
\]

Consequently, when $n\ge 6$, the sequence $\{\gamma_i(n)\}_i$ satisfies the following properties: 
\begin{enumerate}[\quad\rm(1)]
\item $\gamma_1(n)=\gamma_{n-1}(n)=
C_{\ell+2}$;
\item $\gamma_i(n)=\rho_{i+1}(n)$ for all $0\le i\le \ell-1$. 
\end{enumerate}

Furthermore, with $n\ge 6$ and  setting $\gamma_i=\gamma_i(n)$ for all $0\le i\le n$, the Betti table of $A$ over $Q$ is:
\begin{center}
\begin{tabular}{r| c c  c c c c c c c c c c c }
& $0$ & $1$  & $\ldots$  & $\ell-1$ & $\ell$ & $\ell+1$ &$\ell+2$ &$\ell+3$  & $\ldots$  & $n-1$& $n$ \\
\hline
$0$ & $\binom{n}{0}$ & - & $\cdots$ & - & - & - & - & - &$\cdots$ & - & -  \\ 
$1$ & -  & $\binom{n}{1}$   &$\cdots$ &- &- &- &- &- &$\cdots$ &- &- \\ 
% $2$ & $\cdot$ &$\cdot$  &$\binom{n+1}{2}$&$\cdots$ &$\cdot$ &$\cdots$ &$\cdot$ &$\cdot$ &$\cdot$ &$\cdots$ &$\cdot$ &$\cdots$ &$\cdot$ \\ \\
\vphantom{\huge{A}}\vdots&\vdots&\vdots&$\ddots$&\vdots&\vdots&\vdots&\vdots&\vdots&\vdots&\vdots&\vdots\\ 
%\vdots&\vdots&\vdots&\vdots&\vdots&\vdots&\vdots&\vdots&\vdots&\vdots&\vdots&\vdots&\vdots \\ \\ 
$\ell-1$ & -   &- & $\cdots$ &$\binom{n}{\ell-1}$ &- &- &-&-&$\cdots$ & - & - \\
$\ell$& - & $\gamma_1$ &$\ldots$  & $\gamma_{_{\ell-1}}$ & $\gamma_{_{\ell}}$ & $\gamma_{_{\ell+1}}$ &$\gamma_{_{\ell}}$ &$\gamma_{_{\ell-1}}$& $\dots$ &$\gamma_{1}$ & - \\
\vphantom{\huge{A}} $\ell+1$ &- & - & $\cdots$ & - & - & - &- &$\binom{n}{\ell+3}$&$\cdots$ & -& - \\ 
\vphantom{\huge{A}} \vdots&\vdots &\vdots&\vdots&\vdots&\vdots&\vdots&\vdots&\vdots&$\ddots$&\vdots&\vdots\\ 
$n-3$ & -   & -  &$\cdots$ & - & - & - & - & - &$\cdots$ &$\binom{n}{n-1}$ & - \\
$n-2$ & -   & -  &$\cdots$ & - & - & - & - & - &$\cdots$ & - &$\binom{n}{n}$ \\
\end{tabular}
\end{center}
\end{proposition}

\begin{proof} 
Since $A$ is Gorenstein, its minimal free resolution over $Q$ is symmetric. Moreover, since $A$ is Artinian the regularity of $A$ is given by its socle degree, which is $n-2$ by \cref{Hilb R and A}(2$a$).  Using \cref{l: A and R over Pu}(1) and (2) for $u=0$, we have
\begin{equation}
\label{betti_Q A}
\beta_{i,j}^Q(A)=\begin{cases} 
\binom{n}{i},&\ \text{if}\ j = 2i  \text{ and } 0\le i < \ell  \\
0,&\ \text{if}\   j \neq 2i \text{ and }  0\le j-i < \ell\\
0, &\ \text{if}\ j \neq 2i-2 \text{ and } j-i > \ell\\   
\binom{n}{i},&\ \text{if}\ j = 2i-2  \text{ and } i >  \ell+2\,.
\end{cases}
\end{equation}
In particular, when $i\in\{\ell, \ell+1,\ell+2\}$, we have $\beta_{i,j}^Q(A) = 0$ if and only if $j-i\not=\ell$.

It remains to determine the Betti numbers in the $\ell$-th row. 
By symmetry, we have  $\beta_{k, k+\ell}^Q(A)= \beta_{n-k, n-k+\ell}^Q(A)$ for all $0\leq  k\leq \ell+1.$  
We aim to use \cref{betti-strand}, to calculate $\beta_{k, k+\ell}^Q(A)$ for $0\leq k \leq \ell+1$.  To do this, we begin by simplifying the following sum:
\begin{align*}
&\sum_{i=0}^{\ell+k} (-1)^{i+k}\sum_{\substack{j=i\\j-i\ne \ell}}^{k+\ell}\fh_Q(k+\ell-j)\beta_{i,j}^Q(A)\\
= &\sum_{i=0}^{\ell-1} (-1)^{i+k}\sum_{\substack{j=i\\j-i\ne \ell}}^{k+\ell}{\binom{n+k+\ell-j-1}{n-1}}\beta_{i,j}^Q(A)
+ \sum_{i=\ell}^{\ell+2} (-1)^{i+k}\sum_{\substack{j=i\\j-i\ne \ell}}^{k+\ell}\binom{n+k+\ell-j-1}{n-1}\beta_{i,j}^Q(A)\\
&+\sum_{i=\ell+3}^{\ell+k} (-1)^{i+k}\sum_{\substack{j=i\\j-i\ne \ell}}^{k+\ell}\binom{n+k+\ell-j-1}{n-1}\beta_{i,j}^Q(A)\\
%= &\sum_{i=0}^{\ell-1} (-1)^{i+k} \binom{n-1+k+\ell-2i}{n-1}\beta_{i,2i}^Q(A)\\
= &\sum_{i=0}^{\ell-1} (-1)^{i+k}{ \binom{n+k+\ell-2i-1}{n-1}\binom{n}{i}}.
\end{align*}

The first equality above is obtained by splitting the sum in three cases $0\leq i\le\ell-1$, $\ell \leq i\leq \ell+2$, and $\ell+3\leq i\leq \ell+k$ according to \eqref{betti_Q A} and using the equality $\fh_Q(k+\ell+j)=\binom{n+k+\ell-j-1}{n-1}$ for all $k,i\geq 0.$  For the second equality, the first sum simplifies by \eqref{betti_Q A} and the second sum vanishes by the note thereafter. 
%and the fact that when $k\leq \ell+1$, $j$ varies from $i$ to $k+\ell \leq 2\ell+1=n-1$.  
The third sum vanishes because when $i \geq \ell+3$, the only nonzero Betti numbers occur if $j=2i-2 \geq 2\ell+4 = n+2$, but in our case, since $k\leq \ell+1$, we have $j\le k+\ell \leq 2\ell+1=n-1$. 

By \cref{Hilb R and A}(2) we have $\fh_{A}(k+\ell)=\min\left\{ \binom{n}{k+\ell},\binom{n}{k+\ell+2}\right\}=\binom{n}{k+\ell+2}$ and by definition of $Q$ we have $\fh_Q(k-i)=\binom{n+k-i-1}{n-1}$.  
Now substituting $M=A$, $S=Q$, and $a=\ell$ in the  formula of  \cref{betti-strand}, and using the simplified sum above, the symmetry,  and \cref{not: gamma}, we obtain $\beta_{k,k+\ell}^Q(A)=\gamma_k$ for all $k\ge 0$. 

The statements (1) and (2) about the sequence $\{\gamma_i(n)\}_i$ follow from \cref{l: A and R over Pu}, parts (6), respectively (4).  
\end{proof}

\begin{ex} \label{rem:n2n4bettiA}
The Betti tables of $A$ over $Q$ for $n=2$ and $n=4$, respectively, in \cref{betti AQ} are as follows: 
\begin{center}
\begin{tabular}{r| c c c  }
& $0$ & $1$ & $2$  \\
\hline
$0$ & $1$ & $2$ &$1$
\end{tabular}
\qquad 
\begin{tabular}{r| c c c c c }
& $0$ & $1$ & $2$ & $3$  & $4$ \\
\hline
$0$ & $1$ & - & - & -& -\\ 
$1$ & -  & $9$ & $16$ & $9$ & - \\ 
$2$ & - & -  &-&-& $1$  \\ 
\end{tabular}
\end{center}
\end{ex}

\section{Betti numbers of $R$ and $A$ over $Q$}
\label{sec: aci exp}

In this section  we focus on the particular almost complete intersection and Gorenstein rings discussed in the introduction.  Throughout this section, $n\ge 2$ is a fixed integer,  $\kk$ denotes a field of characteristic zero or positive characteristic greater than $n$, and the rings $P$, $R$, and $A$ are as defined in \cref{not:1}, taking $f_i=x_i^2$ for all $0\le i\le n$ and $f_{n+1}=(x_1+\dots+x_n)^2$.  For clarity, we recall this notation with the appropriate substitutions in \cref{not:2} below.       

\begin{notation}
\label{not:2}
Define $Q=\kk[x_1, \dots, x_n]$, where $\kk$ denotes a field of characteristic zero or positive characteristic greater than $n$,  and the ideals:
\begin{equation*}
J \coloneq (x^2_1, \dots, x^2_n),\quad I\coloneq (x^2_1, \dots,x_n^2, (x_1+\dots+x_{n})^2),\quad \text{and}\quad G \coloneq  J : I.
\end{equation*}
We also consider the quotient rings of $Q$:
\begin{equation*}
P \coloneq Q/J,\quad R\coloneq  Q/I,\quad A\coloneq Q/G, \quad \text{and}\quad 
 T\coloneq Q/(x^2_n).
 %,\quad\text{for all}\quad 0\leq u\leq n. 
\end{equation*}
Notice that $A\cong P/\ann_P((x_1+\dots+x_n)^2)$.
%,\, \Pu{0} = Q,\, \Pu{n} =P$, and set $\Pu{n+1}\coloneq R.
We also consider the standard graded rings that correspond to the ones above, but with $n-1$ variables instead of $n$ variables:
\begin{align*}
 \ov Q&\coloneq \kk[x_1\dots,x_{n-1}]&
 \ov R&\coloneq \ov Q/(x_1^2,\dots,x_{n-1}^2,(x_1+\dots+x_{n-1})^2)\\
 \ov P&\coloneq \ov Q/(x_1^2,\dots,x_{n-1}^2)&
 \ov A&\coloneq \ov Q/\big((x_1^2, \dots, x_{n-1}^2) : (x_1+\dots +x_{n-1})^2\big).
 \end{align*}
% Clearly, we have  $R/x_1R\cong \ov R,\ P/x_1P\cong \ov P$, and by \cref{A/x1}(2) we have $A/x_1 A\cong \ov A$.    
\end{notation}

These rings satisfy our assumptions in \cref{not:1}, and in view of \cref{ex: wlp}, they also satisfy the property that $f_{n+1}=(x_1+\dots+x_n)^2$ is a quadratic maximal rank element of $P$, and hence one can apply the results of  \cref{sec:consequences}.  In this section we use these results to
%prove that certain liftability properties hold for $R$ and $A$ and leverage this to 
construct the Betti tables of $R$ and $A$ over $Q$ in the case where $n$ is odd.  Recall that the case where $n$ is even is addressed in  \cref{betti RQ} and \cref{betti AQ}. The results of this section lay the groundwork for an investigation of ideals generated by $n+1$ general quadrics in \cite{aci}.

Now we collect some useful lemmas we will use throughout the section. Recall that $e(R)$ denotes the multiplicity of $R$. 

\begin{lemma}\label{lem:ehRhA bar} 
Adopt the notation and hypothesis in \cref{not:2}. Then: 
\begin{align}
e(R)&={\binom{n+1}{\lfloor\frac{n+1}{2}\rfloor}} \qquad &&e(\ov R)={\binom{n}{\lfloor\frac{n}{2}\rfloor}}\label{eq:e R bar}\\ 
\fh_{R}(i)& = \mathrm{max}\,\left\{{ \binom{n}{i}-\binom{n}{i-2}},0 \right\}\qquad &&\fh_{\ov R}(i)= \mathrm{max}\,\left\{{\binom{n-1}{i}-\binom{n-1}{i-2}},0 \right\}\label{eq:hf R bar}\\
\fh_{A}(i) &=\min\left\{{\binom{n}{i},\binom{n}{i+2}}\right\}\qquad &&\fh_{\ov A}(i)=\min\left\{ {\binom{n-1}{i},\binom{n-1}{i+2}}\right\}.
\end{align}
\end{lemma}
\begin{proof}
It suffices to explain the statements for $A$ and $R$. As noted in \cref{ex: wlp}, $(x_1+\dots+x_n)^2$ is a quadratic maximal rank element of $P$, and hence \cref{Hilb R and A} gives the desired formulas. 
\end{proof}

Towards our goal of computing the Betti tables of $R$ and $A$ over $Q$ when $n$ is odd, we aim to show that both $R$ and $A$ are liftable to $Q$, when regarded as $T$-modules. The following results and notation lay the groundwork for this statement. 

\begin{notation}
\label{not:3}
Let $Q=\kk[x_1, \dots, x_n]$  and consider the ideals:
\begin{equation*}
J' \coloneq (x_1^2-x_n^2, \dots, x_{n-1}^2-x_n^2)\quad \text{and} \quad I'\coloneq J'+ (f),\quad\text{ where } f = (x_1+\dots+x_n)^2-x_n^2.
\end{equation*}
 Consider the quotient rings of $Q$:
\begin{equation*}
P'\coloneq Q/J',\quad R'\coloneq  Q/I',\quad\text{and}\quad A'\coloneq P'/\ann_{P'}(f).
\end{equation*}
\end{notation} 

As in \cref{APR}, we have the following short exact sequence.

\begin{remark}
\label{A'P'R'}
Adopt the notation and hypotheses in \cref{not:3}. There exists a short exact sequence of graded $P'$-modules:
\begin{equation}
\label{ses A'P'R'}
0\to A'(-2)\to P'\to R'\to 0
\end{equation}
where the map $P'\to R'$ is the canonical projection, and the map $A'(-2)\to P$ is multiplication by the element $f$. 
\end{remark}

\begin{lemma} 
\label{x1 reg} Adopt the notation and hypotheses in \cref{not:3}. 
If $n\geq 3$ is odd, then the following assertions  hold: 
\begin{enumerate}[\quad $(1)$]
\item $x_n$ is a regular element of $R'$;
\item $x_n$ is a regular element of $A'$.
\end{enumerate}
\end{lemma}

\begin{proof}
(1): Remark that $\ov R \cong R'/x_n R'$, where $\ov R$ is as defined in \cref{not:2}, and consider the short exact sequence
\begin{equation}
\label{seq}
0\to K\to R'(-1)\xrightarrow{\cdot x_n} R'\to \ov R\to 0,
\end{equation}
where $K$ is the kernel of the map given by multiplication by $x_n$.  To show $x_n$ is regular on $R'$, it suffices to prove that $K=0$. 
 
Let $X$ denote the set of points $p\in \mathbb P^{n-1}(\kk)$ with the property that all entries of $p$ are $\pm 1$ and the number of positive entries of $p$ is equal to one more than the number of negative entries of $p$, where $\mathbb P^{n-1}(\kk)$ is the projective space of dimension $n-1$. Since $n$ is odd, we have that $X$ is nonempty.  We claim that $I'=I_X$, where 
\[I_X\coloneq (g :  g\in \kk[x_1, \dots, x_n]\text{ with } g(p)=0\, \text{ for all }\  p\in X).\] Observe that $V(I')$, the set of zeros of the polynomials in $I'$, is equal to $X$.  In particular, we have $\dim R'=1=\dim (Q/I_X)$ and $I'\subseteq I_X$.

The sequence \eqref{seq} gives an equality of Hilbert series
\begin{equation}
\label{e:hilb}
\Hilb_{R'}(t)(1-t)=\Hilb_{\ov R}(t)-\Hilb_K(t).
\end{equation}
Since $\dim(R')=1$, we can write $\Hilb_{R'}(t)=\frac{q_{R'}(t)}{1-t}$, where $q_{R'}(1)=e(R')$.
For $i$ large, one can see that $\m^i=x_n\m^{i-1}$, where $\m$ denotes the maximal homogeneous ideal of $R'$, and hence $\m^iK=0$. Thus, \eqref{e:hilb} translates into an equality of polynomials: 
\begin{equation*}
q_{R'}(t)=\Hilb_{\overline R}(t)-\Hilb_K(t).
\end{equation*}
Evaluating at $t=1$, we obtain: 
\begin{equation}
\label{eR'}
e(R')\le e(\ov R),
\end{equation}
with equality if and only if $K=0$. 
Recall that $I'\subseteq I_X$, and hence $Q/I_X$ is a homomorphic image of $R'$. In particular, 
\begin{equation}
\label{eIX}
e(Q/I_X)\le e(R').
\end{equation} 
Since the set $X$ of points has cardinality $\binom{n}{\lfloor \frac{n}{2}\rfloor}$, we conclude that $e(Q/I_X)=\binom{n}{\lfloor \frac{n}{2}\rfloor}$. On the other hand, 
the ring $\ov R$ is Artinian with $e(\ov R)=\binom{n}{\lfloor \frac{n}{2}\rfloor}$ by \eqref{eq:e R bar} in \cref{lem:ehRhA bar}. Hence,  equalities  must hold in both \eqref{eR'} and  \eqref{eIX}, so  $K=0$ (and $I'=I_X$).

(2):  It is easy to show that $x_n$ is a regular element of $P'$. By part (1), $x_n$ is regular element of $R'$. The rows of the commutative diagram below are exact by \cref{A'P'R'}.
\[ 
\xymatrixrowsep{1.5pc}
\xymatrixcolsep{1pc}
\xymatrix{
0\ar@{->}[r]^{} 
&A'(-2) \ar@{->}[r]^{} \ar@{->}[d]_{\cdot x_n} 
&P'\ar@{->}[r]^{} \ar@{->}[d]_{\cdot x_n}
&R'\ar@{->}[r]^{} \ar@{->}[d]_{\cdot x_n} 
&0\\
0\ar@{->}[r]^{} 
&A'(-2) \ar@{->}[r]^{} 
&P'\ar@{->}[r]^{}
&R'\ar@{->}[r]^{} 
&0
}
\]
Now applying the Snake Lemma to the diagram  yields the desired conclusion.
\end{proof}

\begin{lemma}
\label{A/x1}
Adopt the notation and hypotheses in \cref{not:2} and \cref{not:3}. If $n\geq 3$ is odd, then there are canonical ring isomorphisms:
\begin{enumerate}[\quad $(1)$] 
\item $A\cong A'/x_n^2A'$;
\item $\overline A\cong  A'/x_nA'\cong A/x_nA$. 
\end{enumerate}
\end{lemma}

\begin{proof} (1): Tensoring the short exact sequence \eqref{ses A'P'R'} in \cref{A'P'R'} with $T$, we obtain an exact sequence of $Q$-modules
\[
\Tor_1^Q(R',T)\to (A'/x_n^2A')(-2)\to P'/x_n^2P'\xrightarrow{\pi'} R'/x_n^2R'\to 0,
\]
where $\pi'$ is the natural projection. 
By \cref{x1 reg}(1), the element $x_n^2$ is regular on $R'$. Thus,  we have $\Tor_1^Q(R',T)=0$. Further, using the natural isomorphisms of $Q$-modules  $P'/x_n^2P'\cong P$ and $R'/x_n^2R'\cong R$ we obtain  an exact sequence
$Q$-modules 
\begin{equation}
\label{eq: A'PR}
0\to (A'/x_n^2A')(-2)\to P\xrightarrow{\pi} R\to 0,
\end{equation}
where $\pi$ is the natural projection and the second map is given by multiplication by $f$, equivalently, multiplication by $(x_1+\dots+x_n)^2$. Comparing \eqref{eq: A'PR} with the short exact sequence \eqref{ses APR} in \cref{APR}, we obtain a canonical isomorphism $A\cong A'/x_n^2A'$ that sends the image of the variable $x_i$ in one ring to the image of the variable $x_i$ in the other ring for each $i$.  

(2): Tensoring the short exact sequence \eqref{ses A'P'R'} in \cref{A'P'R'} with $Q/x_nQ$ and using the same argument, we obtain a short exact sequence 
\[
0=\Tor_1^Q(R',Q/x_nQ)\to A'/x_nA'(-2)\to P'/x_n P'\xrightarrow{\pi''} R'/x_n R'\to 0,
\]
where $\pi''$ is the natural projection.
Observe that there are natural isomorphisms of $Q$-modules $P'/x_n P'\cong \overline P$ and  $R'/x_n R'\cong \overline{R}$. 
Thus, we get a short exact sequence 
\[
0\to (A'/x_nA')(-2)\to \overline P\xrightarrow{\overline \pi} \overline R\to 0,
\]
where $\overline\pi$ is the natural projection and the second map is multiplication by $f$, equivalently, by $(x_1+\dots+x_{n-1})^2$.  Comparing with the exact sequence \eqref{ses APR} in \cref{APR} for the rings $\overline R$, $\overline P$ and $\overline A$, we obtain a canonical isomorphism $\overline A \cong A'/x_nA'$ that sends the image of the variable $x_i$ in one ring to the image of the variable $x_i$ in the other ring for each $i$. This proves the first isomorphism in (2). Using the isomorphism in (1), the second isomorphism in (2) holds as we have
\begin{align*}
A/x_nA&\cong \left(A'/x_n^2A'\right)\Big/ x_n\left(A'/x_n^2A'\right)\cong A'/x_nA'. \qedhere
\end{align*}
\end{proof}

The next result shows that $x_n$ is an exact zero divisor on both $R$ and $A$; see \cite{HenriquesSega11} for the terminology and related results. This is a key ingredient in the proof of \cref{betti reduction}  where we compute the Betti numbers of $R$ and $A$ over $T$.  

\begin{lemma} 
\label{l:R-ezd} Adopt the notation and hypotheses in \cref{not:2}. If $n\ge 3$ is odd, then the following equalities hold:
\begin{enumerate}[\quad $(1)$]
\item $\ann_R(x_n)=x_nR$;
\item $\ann_A(x_n)=x_nA$.
\end{enumerate}
Thus $x_n$ is an exact zero divisor on both $R$ and $A$.
\end{lemma}

\begin{proof}  
(1):  First observe that  $x_nR\subseteq \ann_R(x_n)$.  To establish the desired equality it suffices to prove that the Hilbert functions are equal; that is,  $\fh_{x_nR}(i)=\fh_{\ann_R(x_n)}(i)$ for all $i\geq 0$.
 
 From the short exact sequences
\begin{equation}
\label{ses}
    0 \to \ann_R(x_n) \to R \xrightarrow{\cdot x_n} (x_nR)(1) \to 0\quad\text{and}\quad  0 \to x_nR \to R \xrightarrow{}R/x_nR \to 0
\end{equation}
and the fact that $\ov R\cong R/x_n R$ we get the equalities:
\begin{align}
\label{h eq1}
\fh_{\ann_R(x_n)}(i) &= \fh_R(i) - \fh_{x_nR}(i+1)\\
&=\fh_R(i)-\fh_R(i+1)+\fh_{\ov R}(i+1) \nonumber
\end{align}
for all $i\geq 0$.  We claim that the following equalities hold:
\begin{equation}
\label{h eq2}
   \fh_{\ov R}(i+1)+\fh_{\ov R}(i) = \fh_R(i+1),
\end{equation}
for all $i\geq 0$. Assuming the claim, we obtain the desired equality of Hilbert functions:
\begin{align*}
\fh_{\ann_R(x_n)}(i)
&=\fh_R(i)-\fh_R(i+1)+\fh_{\ov R}(i+1)\\
&= \fh_R(i)-\fh_{\ov R}(i)\\
&=\fh_{x_nR}(i).
\end{align*}
The first equality is \eqref{h eq1}, the second equality follows from  the claim \eqref{h eq2}, and the third equality uses the second short exact sequence of \eqref{ses}.

Now we prove the claim \eqref{h eq2} for a fixed $i\geq 0.$
Using \eqref{eq:hf R bar} in \cref{lem:ehRhA bar}, we need to establish the following equality:
\begin{equation*}
\mathrm{max}\,\left\{  {\binom{n-1}{i+1}-\binom{n-1}{i-1}}, 0 \right\}+\mathrm{max}\,\left\{{ \binom{n-1}{i}-\binom{n-1}{i-2}},0 \right\}=\mathrm{max}\,\left\{  {\binom{n}{i+1}-\binom{n}{i-1}},0 \right\}. 
\end{equation*}
Set $a\coloneq   {\binom{n-1}{i+1}-\binom{n-1}{i-1}}$ and  $b\coloneq{ \binom{n-1}{i}-\binom{n-1}{i-2}}$ and 
notice that $a+b=\binom{n}{i+1}-\binom{n}{i-1}$ using Pascal's identity. Thus,  to prove that $\eqref{h eq2}$ holds, it suffices to establish the following equality:
\begin{equation}
\label{h eq3 first}
    \max\{a,0\}+\max\{b,0\} = \max\{a+b,0\}.
\end{equation}

By hypothesis $n\ge 3$ is odd, so we can write $n=2m+1$ with $m>0$. Considering each of the cases $i<m$, $i=m$, $i=m+1$, $i>m+1$, one can clearly see that the equality \eqref{h eq3 first} holds, which completes the proof. 

(2): The proof is similar to part (1). We need to prove that 
$\fh_{x_nA}(i)=\fh_{\ann_A(x_n)}(i)$ for all $i\geq 0$ which reduces to showing an equality similar to \eqref{h eq2}:
\begin{equation}
\label{h eq3}
   \fh_{\ov A}(i+1)+\fh_{\ov A}(i) = \fh_A(i+1),
\end{equation}
for all $i\geq 0$. In view of \eqref{eq:hf R bar}, we need to show
\[
\min\left\{ {\binom{n-1}{i+1},\binom{n-1}{i+3}}\right\}+\min\left\{ {\binom{n-1}{i},\binom{n-1}{i+2}}\right\}=\min\left\{ {\binom{n}{i+1},\binom{n}{i+3}}\right\}\,.
\]
The verification of this equality can be done as above using Pascal's identity and by considering each of the cases $i<m-1$, $i=m-1$ and $i=m$ and $i>m$ separately.
\end{proof}

\begin{proposition}
\label{betti reduction}
Adopt the notation and hypotheses in \cref{not:2}. If  $n\ge 3$ is odd, then the following equalities hold for all $i,j\ge 0$: 
\begin{enumerate}[\quad $(1)$]
\item  $\beta_{i,j}^{T}(R)=\beta_{i,j}^{\ov Q}(\ov R)$;
\item  $\beta_{i,j}^{T}(A)=\beta_{i,j}^{\ov Q}(\ov A).$
\end{enumerate}
\end{proposition}

\begin{proof} 
(1): Since $T=Q/x_n^2Q$ we have the equality $\ann_{T}(x_n)=x_nT$. Thus, a minimal graded free resolution of  $T/x_nT$ over $T$ has the form:
\[ F_{\bullet}: \quad  \cdots \to T(-2) \xrightarrow{\cdot x_n} T(-1) \xrightarrow{\cdot x_n} T\to T/x_n T\to 0. \]
By \cref{l:R-ezd}(1), we have  $\ann_R(x_n)=x_nR$, hence by tensoring $F_{\bullet}$ with $R$ over $ T$  we also obtain a minimal graded free resolution of $R/x_nR$ over $R$:
\[ F_{\bullet}\otimes_{ T}R: \quad  \cdots \to 
R(-2) \xrightarrow{\cdot x_n} 
R(-1)\xrightarrow{\cdot x_n} R\to R/x_nR\to 0. \] 
In particular, we obtain   
\begin{eqnarray}\label{Tor_Pu{1}}
 \Tor^{ T}_i( T/x_n T, R)=0\quad\text{for all}\quad i>0. 
\end{eqnarray}
If $G_{\bullet}$ is a  minimal graded free resolution of  $R$ over $ T$, then 
\[\HH_i(  T/x_n T\otimes_{ T}G_{\bullet} )= \Tor^{ T}_i( T/x_n T, R)=0 \text{ for all } i>0.\]
Hence, $ T/x_n T\otimes_{ T} G_{\bullet}$ is an acyclic complex, so it is a minimal graded free  resolution of $ T/x_n T\otimes_{ T} R\cong  R/x_nR$ over $ T/x_n T$.
Therefore, we obtain  
\[\beta_{i,j}^{ T}(R)=\beta_{i,j}^{ T/x_n T}(R/x_nR)\quad\text{for all}\quad i,j\geq 0.\] Noticing  that $ T/x_n T\cong Q/x_n Q\cong\overline Q$,  the desired conclusion follows.

(2): Using \cref{l:R-ezd}(2) instead of \cref{l:R-ezd}(1), the proof follows as in  (1).
\end{proof}

\chunk
\label{exact-sequence}
Let $S$ be a standard graded Noetherian ring, $z$ a homogeneous regular element of $S$  of degree $d$, and set $\overline S\coloneq S/zS$.  Then, for finitely generated graded $\overline S$-modules $M$, $N$, the mapping cone of the {\it Eisenbud operator} $\chi_{i,j}$ defined by $z$ (see \cite{eisen}) gives the exact sequence, for all $i,j\in\BZ$:
\begin{equation*}
\dots \to \Tor_i^{S}(M,N)_j\to \Tor_i^{\overline S}(M,N)_j\xrightarrow{\chi_{i,j}}\Tor_{i-2}^{\overline S}(M,N)_{j-d}\to \Tor_{i-1}^{S}(M,N)_j\to \dots
\end{equation*}

A graded $\overline S$-module $M$ is said to be {\it liftable to $S$} if there exists a graded $S$-module $M'$ such that $M\cong M'\otimes_S \overline S$ and $\Tor_i^S(M',\overline S)=0$ for all $i>0$.
If $M$ is liftable to $S$ then for all $i,j\in\BZ$  the map $\chi_{i,j}$ is zero for all graded $\overline S$-modules $N$; see for example \cite[Theorem 3.1]{Dao}. 

\begin{proposition}
\label{t:RAoverQ} 
Adopt the notation and hypotheses in \cref{not:2}. If $n\geq 3$ is odd, then the $T$-modules $R$ and $A$ are liftable to $Q$, and hence the following equalities hold: 
\begin{align*}
\beta_{i,j}^Q(R)&=\beta_{i,j}^{ T}(R)+\beta_{i-1,j-2}^{ T}(R)=\beta_{i,j}^{\overline Q}({\overline R})+\beta_{i-1,j-2}^{\overline Q}({\overline R});\\ \beta_{i,j}^Q(A)&=\beta_{i,j}^{ T}(A)+\beta_{i-1,j-2}^{ T}(A)=\beta_{i,j}^{\overline Q}(\overline A)+\beta_{i-1,j-2}^{\overline Q}({\overline A}).
\end{align*} 
\end{proposition}

\begin{proof}  
Adopting \cref{not:3}, it is easy to see that $R\cong R'/x_n^2R'\cong R'\otimes_Q T$. By \cref{x1 reg}(1) the element $x_n^2$ is regular on $R'$. Taking $M'=R'$ and $z=x_n^2$ in \cref{exact-sequence}, we have that the $T$-module $R$ is liftable to $Q$ and we get the following short exact sequences for all $i,j\geq 0$:
\begin{equation*}
 0\to\Tor_{i-1}^{ T}(R,\kk)_{j-2}\to \Tor_{i}^{Q}(R,\kk)_j\to \Tor_{i}^{ T}(R,\kk)_{j}\to 0\,.
\end{equation*} 
This sequence, together with \cref{betti reduction}(1), yields the desired equalities.

By \Cref{A/x1}(1) we have $A\cong A'/x_n^2A'$. 
Now a similar proof using \cref{x1 reg}(2) and
\cref{betti reduction}(2) gives the desired equalities of Betti numbers of $A$.
\end{proof}

\begin{theorem}\label{c:RoverQ betti}
Let $n\geq 3$ be an odd integer and let $\ell=\frac{n-3}{2}$.  Let $Q=\kk[x_1, \dots, x_n]$, where $\kk$ is a field with $\ch\kk>n$ or $\ch\kk=0$, and consider the ring
\begin{equation*}
R = Q/\left(x^2_1, \dots, x^2_n, (x_1+\dots+ x_n)^2\right)\,.
\end{equation*} 
The Betti numbers of $R$ as a $Q$-module are:
\[
\beta_{i,j}^Q(R)=\begin{cases}
\binom{n+1}{i},&\text{ if } j=2i\,  \text{ and }\,  0\leq i \leq\ell\\
\rho_{\ell+1}(n-1)+\binom{n}{\ell},&\text{ if } j=2i=i+(\ell+1)\\
\rho_i(n-1),& \text{ if } j = i+(\ell+1)\ne 2i \text{ and }\,  0\leq i \leq n-1\\
\rho_{i-1}(n-1),& \text{ if } j = i+(\ell+2)  \text{ and }\,  1\leq i \leq n\\
0,&\text{ otherwise}\,, \end{cases}
\]
where the sequence $\{\rho_i(n-1)\}$ is defined as in \cref{not: rho}. In particular, for $n\geq 7$ and with $\rho_k=\rho_k(n-1)$ for all $0\le k\le n-1$, the Betti table of $R$ as a $Q$-module is: 
\begin{center}
\setlength{\tabcolsep}{3pt} 
\begin{tabular}{r| c c c c c c c c c c c c c c }
& \small{$0$} & \small{$1$} & \small{$2$} & \small{$3$}  & $\ldots$ & \small{$\ell$} & \small{$\ell+1$} & \small{$\ell+2$} & \small{$\ell+3$} & \small{$\ell+4$} & $\ldots$ & \small{$n-1$} & \small{$n$} \\
\hline
$0$ & $\binom{n+1}{0}$ & - & - & - & $\cdots$  & - & - & - & 
- & - & $\cdots$ & - &-\\ 
$1$ & -  & $\binom{n+1}{1}$ & - & - &$\cdots$ & - & - & - & - & - & $\cdots$ & -  & - \\ 
\vphantom{\huge{A}}$\vdots$&$\vdots$&$\vdots$&$\vdots$&$\vdots$&$\ddots$&$\vdots$&$\vdots$&$\vdots$&$\vdots$ & $\vdots$&$\vdots$&$\vdots$&$\vdots$ \\ 
\vphantom{\huge{A}} $\ell$ &  - & - & - & - & - &$\binom{n+1}{\ell}$ &- &- & - & - & $\cdots$ & - & - \\
\vphantom{\huge{A}}$\ell+1$ & - & - & $\rho_{2}$ &$\rho_3$ &$\cdots$ & $\rho_{_{\ell}}$ & $\rho_{_{\ell+1}}+\binom{n}{
\ell}$ & $\rho_{_{\ell+2}}$ & $\rho_{_{\ell+3}}$ & $\rho_{_{\ell}}$ &$\ldots$ & $\rho_{2}$ & - \\
$\ell+2$& - & - & - & $\rho_2$ &$\cdots$ &$\rho_{_{\ell-1}}$& $\rho_{_{\ell}}$ & $\rho_{_{\ell+1}}$ & $\rho_{_{\ell+2}}$ & $\rho_{_{\ell+3}}$ & $\cdots$ &  $\rho_{3}$ &$\rho_{2}$ \\
\end{tabular}
\end{center}
and the sequence $\{\rho_k\}_{k\ge 0}$ satisfies the property: 
$\rho_{\ell+1}+ \binom{n}{\ell} = \rho_{\ell+3}+\binom{n+1}{\ell+1}$.
\end{theorem}

%\oana{To put in the tables using Macaulay 2.}
\begin{proof}
Noting that $\ell=\left\lfloor{\frac{(n-1)-2}{2}}\right\rfloor$ and using \cref{betti RQ} for  the ring $\ov R$, we have  
\[
\beta_{i,j}^{\overline Q}(\overline R)=
\begin{cases} 
\binom{n}{i},&\text{ if } j=2i\,  \text{ and }\,  0\leq i \leq\ell,\\
\rho_i(n-1), &\text{ if } j= i+\ell+1 \text{ and }\,  0\leq i\leq n-1,\\
0,&\text{ otherwise. }
\end{cases}
\]
Using this formula, \cref{t:RAoverQ} yields the desired equalities for all $i,j\ge 0$:
\begin{align*}
\beta_{i,j}^Q(R)&=\beta_{i,j}^{\overline Q}(\overline R)+\beta_{i-1,j-2}^{\overline Q}(\overline R)\\
=& \begin{cases} 
\binom{n}{i},&\text{if } j=2i\,  \text{ and }\,  0\le i \leq\ell\\
\rho_i(n-1), &\text{if } j= i+\ell+1 \text{ and }\,  0\le i\leq n-1\\
0,&\text{ otherwise}
\end{cases}\\
&+\begin{cases} 
\binom{n}{i-1},&\text{if } j-2=2(i-1)\,  \text{ and }\,  0\leq i-1 \leq\ell\\
\rho_{i-1}(n-1), &\text{if } j-2= i-1+\ell+1 \text{ and }\,  0\leq i-1\leq n-1\\
0,&\text{ otherwise}
\end{cases}\\
=&\begin{cases}
\binom{n+1}{i},&\text{ if } j=2i\,  \text{ and }\,  0\leq i \leq\ell\\
\rho_{\ell+1}(n-1)+\binom{n}{\ell},&\text{ if } j=2i=i+(\ell+1)\\
\rho_i(n-1),& \text{ if } j = i+(\ell+1)\ne 2i \text{ and }\,  0\leq i \leq n-1\\
\rho_{i-1}(n-1),& \text{ if } j = i+(\ell+2)  \text{ and }\,  1\leq i \leq n\\
0,&\text{ otherwise}\,.
\end{cases}  
\end{align*}
This computation, together with the properties of the sequence $\{\rho_k(n-1)\}$ from \cref{betti RQ}, give the remaining conclusions.  In particular, using \cref{betti RQ}(4), we see  
\[
\rho_{\ell+1}(n-1) = \rho_{\ell+3}(n-1) + \binom{(n-1)+1}{\ell+1} = \rho_{\ell+3}(n-1) + \binom{n}{\ell+1}.
\] 
Adding $\binom{n}{\ell}$ to both sides of this equality yields the desired property of the sequence $\{\rho_k\}_{k\ge 0}$
\end{proof}

\begin{ex} \label{rem:n3n5bettiR}
The Betti tables of $R$ over $Q$ for $n=3$ and $n=5$, respectively, in \cref{c:RoverQ betti} are as follows: 
\begin{center}
\begin{tabular}{r| c c c c }
& $0$ & $1$ & $2$ & $3$  \\
\hline
$0$ & $1$ & - & - & -\\ 
$1$ & -  & $4$ & $2$ & - \\ 
$2$ & - & -  & $3$ &$2$ \\ 
\end{tabular}\qquad 
\begin{tabular}{r| c c c c c c}
& $0$ & $1$ & $2$ & $3$  & $4$ & $5$ \\
\hline
$0$ & $1$ & - & - & -& -&-\\ 
$1$ & -  & $6$ & - & - & - &-\\ 
$2$ & - & -  &$20$ &$16$  &$5$ &- \\ 
$3$ & - & -  &- &$15$  &$16$ &$5$ 
\end{tabular}
\end{center}
\end{ex}

\begin{theorem}\label{c:AoverQ betti}
Let $n\geq 3$ be an odd integer and let $\ell=\frac{n-3}{2}$.  Let $Q=\kk[x_1, \dots, x_n]$, where $\kk$ is a field with $\ch\kk>n$ or $\ch\kk=0$, and consider the ring
\begin{equation*}
A = Q/\left( (x^2_1, \dots, x^2_n) :(x_1+\dots+ x_n)^2\right).
\end{equation*}
The Betti numbers of $A$ as a $Q$-module are 
\[
\beta_{i,j}^Q(A)=\begin{cases} 
\binom{n}{i},&\text{ if } j = 2i  \text{ and } 0\leq  i \leq \ell-1 \text{ or  if } j = 2i-2 \text{ and }  \ell+4 \leq i\leq n\\
\gamma_{\ell}(n-1)+\binom{n-1}{\ell-1},&\text{ if } j=2i=i+\ell\\
\gamma_{\ell+2}(n-1)+\binom{n-1}{\ell+3},&\text{ if } j=2i-2=i+(\ell+1)\\
\gamma_i(n-1),&\text{ if }\ j=i+\ell\ne 2i \text{ and }\,  0\leq i\leq n-1\\
\gamma_{i-1}(n-1),&\text{ if } j=i+\ell+1\ne 2i-2 \text{ and }\,  1\leq i\leq n\\
0,&\text{ otherwise}\,,
\end{cases}
\]
where the sequence  $\{\gamma_i(n-1)\}$ is defined as in \cref{not: gamma}.

In particular, when $n\ge 7$ and with $\gamma_i=\gamma_i(n-1)$ for all $0\le i\le n-1$,  the Betti table of $A$ as a $Q$-module is:

\begin{center}
\setlength{\tabcolsep}{3.5pt} 
\begin{tabular}{r| c c c  c c c c c c c c c c c c c}
& \footnotesize$0$ & \footnotesize$1$  & \footnotesize$2$ &$\cdots$  & \footnotesize$\ell-1$ & \footnotesize$\ell$ & \footnotesize$\ell+1$ &\footnotesize$\ell+2$ &\footnotesize$\ell+3$ &\footnotesize$\ell+4$ & $\cdots$ &\footnotesize$n-2$ & \footnotesize$n-1$& \footnotesize$n$ \\
\hline
\footnotesize\footnotesize\footnotesize$0$ & \footnotesize$\binom{n}{0}$ & - & - & $\cdots$ & - & - & - & - &-&-&$\cdots$ & - & - &- \\ 
\footnotesize\footnotesize$1$ & -  &\footnotesize $\binom{n}{1}$ &-   &$\cdots$ &- &- &- &- &- &-&$\cdots$ &- &- &-\\ 
% \footnotesize$2$ & - & -  &\footnotesize$\binom{n}{2}$&$\cdots$ &- &- &-&-&- &-&$\cdots$ &- &- &- \\ 
\vphantom{\huge{I}}$\vdots$&$\vdots$&$\vdots$&$\vdots$&$\ddots$&$\vdots$&$\vdots$&$\vdots$&$\vdots$&$\vdots$&\vdots &$\ddots$&$\vdots$&$\vdots$&$\vdots$\\ 
%\vdots&\vdots&\vdots&\vdots&\vdots&\vdots&\vdots&\vdots&\vdots&\vdots&\vdots&\vdots&\vdots \\ \\ 
\vphantom{\huge{I}}\footnotesize$\ell-1$ & - &-  &- & $\cdots$ &\footnotesize$\binom{n}{\ell-1}$ &- &- &-&-&-&$\cdots$ & - & - &-\\
\footnotesize$\ell$& - & $\gamma_{_1}$  &$\gamma_{_2}$&$\cdots$  & $\gamma_{_{\ell-1}}$ & $\gamma_{_{\ell}}$\footnotesize$+\binom{n-1}{\ell-1}$ & $\gamma_{_{\ell+1}}$ &$\gamma_{_{\ell}}$ &$\gamma_{_{\ell-1}}$& $\gamma_{_{\ell-2}}$& $\cdots$ &$\gamma_{_{1}}$ & -&-\\
\vphantom{\huge{I}}\footnotesize$\ell+1$& - &- & $\gamma_{_1}$ &$\cdots$  & $\gamma_{_{\ell-2}}$ & $\gamma_{_{\ell-1}}$ & $\gamma_{_{\ell}}$ &$\gamma_{_{\ell+1}}$ &$\gamma_{_{\ell}}$\footnotesize$+\binom{n-1}{\ell-1}$ &$\gamma_{_{\ell-1}}$& $\cdots$ &$\gamma_{_{2}}$ &$\gamma_{_{1}}$&- \\
\vphantom{\huge{I}}\footnotesize$\ell+2$ &- & - & -&$\cdots$ & - & - &- &-&-&\footnotesize$\binom{n}{\ell+4}$&$\cdots$ & -& -&- \\ 
\vphantom{\huge{I}} \vdots&\vdots &\vdots&\vdots&$\ddots$&\vdots&\vdots&\vdots&\vdots&\vdots&$\vdots$&$\ddots$&\vdots&\vdots&\vdots\\ 
\footnotesize$n-3$ & -   & -  &- &$\cdots$ & - & - & - & - & -&- &$\cdots$ &-&\footnotesize$\binom{n}{n-1}$ & - \\
\footnotesize$n-2$ & -   & - &- &$\cdots$ & - & - & - & - & - &-&$\cdots$ & - &-&\footnotesize$\binom{n}{n}$ \\
\end{tabular}
\end{center}
and $\gamma_1=C_{\ell+2}$. 
\end{theorem}

\begin{proof}
Using \cref{betti AQ} for  the ring $\ov A$, we have
\[
\beta_{i,j}^{\overline Q}(\overline A)=
\begin{cases} 
\binom{n-1}{i},&\ \text{if}\  j = 2i \text{ and }  i \le  \ell-1     \text{ or  if } j = 2i-2  \text{ and } i \ge  \ell+3    \\
\gamma_i(n-1),&\ \text{if}\ j=i+\ell\text{ and }\,  0\leq i\leq n-1\\
0,&\ \text{otherwise\,.} 
\end{cases}
\]
This formula and \cref{t:RAoverQ} yield the following equalities of Betti numbers, for all $i, j\ge 0$:
\begin{align*}
\beta_{i,j}^Q(A)=&\beta_{i,j}^{\overline Q}(\overline A)+\beta_{i-1,j-2}^{\overline Q}(\overline A)\\
=&\begin{cases} 
\binom{n-1}{i},&\text{ if } j = 2i   \text{ and } 0\leq  i \leq \ell-1 \text{ or  if }\ j = 2i-2   \text{ and }  i \geq  \ell+3\\
\gamma_i(n-1),&\text{ if}\ j=i+\ell \text{ and }\,  0\leq i\leq n-1\\
0,&\text{ otherwise}
\end{cases}\\
&+\begin{cases} 
\binom{n-1}{i-1},&\text{ if } j= 2i   \text{ and } 1\leq  i \leq \ell \text{ or  if }\ j= 2i-2   \text{ and }  i\geq  \ell+4\\
\gamma_{i-1}(n-1),&\text{ if }\ j=i+\ell+1 \text{ and }\,  1\leq i\leq n\\
0,&\text{ otherwise}
\end{cases}\\
=&\begin{cases} 
\binom{n}{i},&\text{ if } j = 2i  \text{ and } 0\leq  i \leq \ell-1 \text{ or  if } j = 2i-2 \text{ and }  i \geq  \ell+4\\
\gamma_{\ell}(n-1)+\binom{n-1}{\ell-1},&\text{ if } j=2i=i+\ell\\
\gamma_{\ell+2}(n-1)+\binom{n-1}{\ell+3},&\text{ if } j=2i-2=i+(\ell+1)\\
\gamma_i(n-1),&\text{ if }\ j=i+\ell\ne 2i \text{ and }\,  0\leq i\leq n-1\\
\gamma_{i-1}(n-1),&\text{ if } j=i+\ell+1\ne 2i-2 \text{ and }\,  1\leq i\leq n\\
0,&\text{ otherwise}.
\end{cases}
\end{align*}
Recalling that $\gamma_0(n-1)=0$, this computation gives the desired Betti table of $A$. The fact that $\gamma_1(n-1)=C_{\ell+2}$ follows from \cref{betti AQ}(1).
\end{proof}

\begin{ex} \label{rem:n3n5bettiA}
The Betti tables of $A$ over $Q$ for $n=3$ and $n=5$,  respectively, in \cref{c:AoverQ betti} are as follows: 
\begin{center}
\begin{tabular}{r| c c c c }
& $0$ & $1$ & $2$ & $3$ \\
\hline
$0$ & $1$ & $2$ & $1$ & -\\ 
$1$ & -  &$1$ &$2$  &$1$ 
\end{tabular}
\qquad 
\begin{tabular}{r| c c c c c c}
& $0$ & $1$ & $2$ & $3$  & $4$ & $5$ \\
\hline
$0$ & $1$ & - & - & -& -&-\\ 
$1$ & -  &$10$ &$16$  &$9$ &- &-\\ 
$2$ & - & -   &$9$  &$16$ &$10$&- \\ 
$3$ & -  & - & - & - & - & $1$
\end{tabular}
\end{center}
\end{ex}

\begin{corollary}
\label{Cataln-socle-gens}
Let $n\geq 2$ be an integer and let $\ell=\left\lfloor\frac{n-2}{2}\right\rfloor$. Let $Q=\kk[x_1, \dots, x_n]$, where $\kk$ is a field with $\ch\kk>n$ or $\ch\kk=0$, and consider the ring
%Let $n\geq 2$ be an integer, $\ell=\lfloor\frac{n-2}{2}\rfloor$, $\kk$ a field with $\ch\kk>n$ or $\ch\kk=0$,  $Q=\kk[x_1, \dots, x_n]$,  and consider the rings
$R=Q/I$ and the ideal $G$, where
\begin{equation*}
    J=(x^2_1, \dots, x^2_n),\quad I=J+(x_1+\dots+ x_n)^2,\quad\text{and}\quad G=J:I\,.
\end{equation*}

\begin{enumerate}[\quad\rm(1)]
\item The ring $R$ is level, with socle in degree $n-\ell-1$, and
\[
\dim_\kk(\Soc(R))=C_{\ell+2}.
\]
\item The ideal $G/J$ is generated in degree $\ell+1$ and the minimal number of generators of $G/J$ is equal to $C_{\ell+2}$. Consequently, the ideal $G$ is generated in degrees $2$ and $\ell+1$ and, when $n\ge 4$ (and thus $\ell\ge 1$), its minimal number of generators is $n+C_{\ell+2}$. 
\end{enumerate}
\end{corollary}

\begin{proof}
When $n\ge 6$ is even, the statements follow directly from the Betti tables  and the listed properties in \cref{betti RQ} and \cref{betti AQ}; see \cref{rem:n2n4bettiR} and \cref{rem:n2n4bettiA} for the cases $n=2$ and $n=4$.
  % When $n$ is even, the statements follow directly from \cref{betti RQ} and \cref{betti AQ}.

When $n\ge 7$ is odd, the statements follow directly from the Betti table  and the listed properties in \cref{c:RoverQ betti} and \cref{c:AoverQ betti}; see \cref{rem:n3n5bettiR} and \cref{rem:n3n5bettiA} for the cases $n=3$ and $n=5$.
\end{proof}

\begin{remark}
Using the results presented here, we subsequently show in \cite{aci} that the conclusions of \cref{Cataln-socle-gens} hold more generally, namely when the ring $R$ is defined by $n+1$ general quadrics. 
\end{remark}

\section{A further study of the Gorenstein ideal $G$}
\label{G}
In this section we keep the notation in \cref{not:2}. We find a set of minimal generators of $G$ and the initial ideal of $G$ and prove that the Gorenstein ring $A$ has the Strong Lefschetz Property when $\ch\kk=0$. Note that the initial ideal of $I$ is computed in \cite{BV} and \cite{kling}.

\subsection*{The generators and initial ideal of $G$}
In what follows, we consider the action of the symmetric group $\sym_n$ on $Q$ defined by permuting the variables, namely
\[
\sigma\cdot f(x_1, \dots, x_n)=f(x_{\sigma(1)}, \dots,x_{\sigma(n)})
\]
where $\sigma\in \sym_n$ and $f\in Q$. We denote by $(f)_{\sym_n}$ the ideal
\[
(f)_{\sym_n}=(\sigma\cdot f\,\,\vert\,\, \sigma\in \sym_n).
\]
Below, we use the reverse lexicographic order $>$ on $Q$, with $x_1>x_2>\dots >x_n$. 

\begin{theorem}\label{thm: G generators}
Let $n\ge 2$ be an integer and let $\ell=\left\lfloor\frac{n-2}{2}\right\rfloor$. Let $Q=\kk[x_1,\dots,x_n]$, where $\kk$ is a field with $\ch\kk>n$ or $\ch\kk=0$, and consider the ideals
%Adopt \cref{not:2} and set $\ell=\lfloor\frac{n-2}{2}\rfloor$.
\begin{equation*}
J = (x^2_1, \dots, x^2_n) \quad \text{and}\quad G = J : (x_1+\dots+ x_n)^2.
\end{equation*}
The Gorenstein ideal  $G$ can be described as follows: 
\begin{align}\label{eq:gensG}
G = \begin{cases}
J+\left((x_1-x_2)(x_3-x_4)\,\cdots\, (x_{n-2}-x_{n-1})\right)_{\sym_n}, &\text{if $n$ is odd}\\
J+\left((x_1-x_2)(x_3-x_4)\,\cdots\, (x_{n-3}-x_{n-2})\,x_{n-1}\right)_{\sym_n},   &\text{if $n$ is even}
\end{cases}
\end{align}
and the initial ideal of $G$ is 
\[
\init_{>}(G)=J+\big(x_{i_1}x_{i_2}\dots x_{i_{\ell+1}}\,\,\vert\, \, 0<i_1<\dots<i_{\ell+1}, \, i_j\le 2j\text{ for all $1\leq j\leq \ell+1$}\big)\,.
\]
In particular, $G$ has a Gr\"obner basis generated in the same degrees as $G$. 
\end{theorem}

\begin{proof}
Let $G'$ denote the ideal on the right-hand side of the equation \eqref{eq:gensG}. First, we show that $G' \subseteq G$. Set
\[
g_n=\begin{cases}
(x_1-x_2)(x_3-x_4)\cdots (x_{n-2}-x_{n-1}),&\text{if $n$ is odd}\\(x_1-x_2)(x_3-x_4)\cdots (x_{n-3}-x_{n-2})x_{n-1}, &\text{if $n$ is even}
\end{cases}
\]
and set $h_n=x_1+\dots+x_n$ and $p_n=(h_n)^2$.  Since $G=J : (p_n)$ and $G'=J+(g_n)_{\sym_n}$, the inclusion $G'\subseteq G$ holds if and only if $(\gamma\cdot g_n)p_n\in J$ for all $\gamma\in \sym_n$. 

Let $\sigma\in \sym_n$. Since $p_n$ is a symmetric polynomial, notice that we have equalities 
\begin{equation*}
\sigma\cdot (g_np_n)=(\sigma\cdot g_n)(\sigma\cdot p_n)=(\sigma\cdot g_n)p_n\,.
\end{equation*}
Thus the inclusion $G'\subseteq G$ holds if and only if $\gamma\cdot(g_np_n)\in J$ for all $\gamma\in \sym_n$. Also, since $J$ is a symmetric ideal, observe that
\begin{equation*}
g_np_n\in J\iff \sigma\cdot (g_np_n)\in J\iff \gamma\cdot (g_np_n)\in J \quad \text{for all}\quad  \gamma\in \sym_n\,.
\end{equation*}

Putting together the observations above, we have thus:
\begin{equation}
\label{sigma}
G'\subseteq G \iff g_np_n\in J\iff
(\sigma\cdot g_n)p_n\in J\quad\text{\!for some $\sigma\in\sym_n$}\,.
\end{equation}

Now we use induction on $n$ to show that $g_np_n\in J=(x_1^2 \dots, x_n^2)$, handling the case where $n$ is even and the case where $n$ is odd simultaneously. 

When $n$ is even, the base case is $n=2$. In this case, $g_2=x_1$ and it is clear that
\begin{align*}
 g_2p_2=x_1(x_1+x_2)^2\in (x_1^2,x_2^2)\,.
\end{align*}
When $n$ is odd, the base case is $n=3$. In this case, $g_3=x_1-x_2$ and it is clear that
\[
g_3p_3=(x_1-x_2)(x_1+x_2+x_3)^2\in (x_1^2, x_2^2,x_3^2)\,.
\]
For the inductive step, assume that $g_{n-2}p_{n-2}\in (x_1^2, \dots, x_{n-2}^2)$. Observe that there exists a permutation $\sigma\in \sym_n$ such that 
\[
\sigma\cdot g_n=g_{n-2}(x_{n-1}-x_n).
\]
In view of \eqref{sigma}, it suffices to show $(\sigma\cdot g_n)p_n\in J=(x_1^2, \dots, x_n^2)$. Indeed, we have:
\begin{align*}
  (\sigma\cdot g_n)p_n&=g_{n-2}(x_{n-1}-x_n)p_n\\
  &=g_{n-2}(x_{n-1}-x_n)(h_{n-2}+x_{n-1}+x_n)^2 \\
  &=g_{n-2}(x_{n-1}-x_n)\left(p_{n-2}+(x_{n-1}+x_n)(2h_{n-2}+x_{n-1}+x_n)\right)\\
  &=g_{n-2}p_{n-2}(x_{n-1}-x_n)+g_{n-2}(x_{n-1}-x_n)(x_{n-1}+x_n)(2h_{n-2}+x_{n-1}+x_n).
\end{align*}
In the sum above, the first summand is in $J$ since $g_{n-2}p_{n-2}\in (x_1^2, \dots, x_{n-2}^2)$ by the induction hypothesis. The second summand is in $J$ since $(x_{n-1}-x_n)(x_{n-1}+x_n)\in(x_{n-1}^2, x_n^2)$. Hence $(\sigma\cdot g_n)p_n\in J=(x_1^2, \dots, x_n^2)$. Thus $G'\subseteq G$ by \eqref{sigma}. 

Observe from \eqref{eq:gensG} that $G'$ is generated in degrees $2$ and $\ell+1$ and the same is true for $G$ by \cref{Cataln-socle-gens}.  Thus we have that  $G'_i=G_i=J_i$ for $i<\ell+1$ and neither $G'$ nor $G$ have minimal generators in degrees higher than $\ell+1$.  Hence, to show $G'=G$ it suffices to show that $\dim_\kk G'_{\ell+1}=\dim_\kk G_{\ell+1}$. In view of the inclusion $G'\subseteq G$ proved above, we know the inequality $\dim_\kk G'_{\ell+1}\le \dim_\kk G_{\ell+1}$. Thus to show $G'=G$, it suffices to show the reverse inequality, which is equivalent to 
\begin{equation}
\label{GG'}
\fh_{Q/G'}(\ell+1)\le \fh_{Q/G}(\ell+1).
\end{equation}
We set $\mathcal G=J+\mathcal G'$, where
\begin{align*}
\mathcal G'&=\big(x_{i_1}x_{i_2}\dots x_{i_{\ell+1}}\,\,\vert\, \, 0<i_1<\dots<i_{\ell+1}, \, i_j\le 2j\text{ for all $1\leq j\leq \ell+1$}\big).
\end{align*}
Now we claim the following inclusion and equality of Hilbert functions:
\begin{align}
    \mathcal G&\subseteq \init_>(G') \label{claim1}\\
    \fh_{Q/\mathcal G}(\ell+1)&=\fh_{Q/G}(\ell+1).\label{claim2}
\end{align}
Assuming the claim, we have the following sequence of (in)equalities:
\begin{align}
\fh_{Q/G}(\ell+1)&=\fh_{Q/\mathcal G}(\ell+1)\nonumber\\
&\ge \fh_{Q/\init_{>}(G')}(\ell+1)\label{sequence-ineq}\\
&=\fh_{Q/G'}(\ell+1).\nonumber
\end{align}
where the first equality follows from \eqref{claim2}, the second inequality follows from \eqref{claim1}, and the third equality is well-known (see for example \cite{Eis}).  In view of \eqref{GG'}, we conclude that $G=G'$. Furthermore, this equality also implies that the inequality in \eqref{sequence-ineq} becomes an equality, and hence $\mathcal G=\init_{>}(G')=\init_{>}(G)$. 

Now we prove the claim. It is known (see for example \cite[Exercise~77]{stanley-catalan}) that the number of sequences of integers  $0<i_1<\dots<i_{\ell+1}$ with $i_j\le 2j$ for all $1\leq j\leq \ell+1$ is equal to the Catalan number $C_{\ell+2}$  and hence 
\begin{equation}\label{e:G/Jcatalan}
    \dim_\kk\mathcal (\mathcal G/J)_{\ell+1}=\dim_\kk\mathcal (\mathcal G')_{\ell+1}=C_{\ell+2}\,.
\end{equation}
Now we have the following equalities:
\begin{align*}
\fh_{Q/\mathcal G}(\ell+1)=\fh_P(\ell+1)-\fh_{\mathcal{G}/J}(\ell+1) 
=\binom{n}{\ell+1}-C_{\ell+2}
=\binom{n}{\ell+3}
=\fh_{Q/G}(\ell+1),
\end{align*}
where the first equality follows from the fact that $Q/\mathcal G\cong P/(\mathcal G/J)$, the second follows from \eqref{e:G/Jcatalan}, the third follows from Pascal's identity and symmetry of binomial coefficients in the even and odd cases, and the last equality follows from \cref{lem:ehRhA bar}.  This completes the proof of \eqref{claim2}.

It remains to show the inclusion $\mathcal G\subseteq \init_{>}G'$ in \eqref{claim1}. Let $0<a_1<\dots<a_{\ell+1}$ with  $a_j\le 2j$ for all $1\le j\le \ell+1$. We want to construct a polynomial in $G'$ whose leading term is equal to $x_{a_1}x_{a_2}\dots x_{a_{\ell+1}}$. 

First assume that $n=2k$ is even. Then $k=\ell+1$.
Let $\{1,\dots,n\}\smallsetminus \{a_1, \dots, a_k\}=\{b_1, \dots, b_k\}$ where $b_1<b_2<\cdots <b_k$ and consider
\[
p=(x_{a_1}-x_{b_2})(x_{a_2}-x_{b_3})\cdots (x_{a_{k-1}}-x_{b_{k}})x_{a_k}.
\]
We claim that the leading term of $p$ is $x_{a_1}x_{a_2}\dots x_{a_{k}}$. Indeed, the terms in $p$ have the form $x_{i_1}x_{i_2}\dots x_{i_{k-1}}x_{a_k}$ where $i_j\in \{a_j,b_{j+1}\}$ for all $1\le j\le k-1$. To see that all these terms are less than $x_{a_1}x_{a_2}\dots x_{a_{k}}$ in the reverse lexicographic order, it suffices to show that $a_j<b_{j+1}$ for all $1\le j\le k-1$. We prove this by induction on $j$. 

For the base case, assume that $j=1$. By construction, it follows that $0<a_1\le 2$. If $a_1=1$, it is clear that $a_1<b_2$. If $a_1=2$, it follows that  $b_1=1$. Since $b_1<b_2$ and $a_1\ne b_2$ we must have $b_2\ge 3$.  Thus $a_1<b_2$ in either case.

Now assume $j\ge 2$ and  $a_{j-1}<b_{j}$. Observe that we have the following inequalities
\begin{align*}
b_j>a_{j-1}>a_{j-2}>\dots >a_1\quad\text{and}\quad
b_j>b_{j-1}>b_{j-2}>\dots>b_1,
\end{align*}
and hence $b_j$ is greater than $2(j-1)$ distinct integers in the set $\{1,\dots,n\}$. Therefore $b_j>2(j-1)$ and thus $b_j\ge 2j-1$. Since $b_{j+1}>b_j$ we have inequalities
\[
a_j \le 2j\le b_{j+1}.
\]
%$b_{j+1}\ge 2j$. Since $a_j\le 2j$ and 
Since $a_j\ne b_{j+1}$, we conclude that $a_j<b_{j+1}$, which completes induction and thus the proof of \eqref{claim1} when $n$ is even. 

Now assume that $n=2k+1$ is odd. Then $k=\ell+1$ as in the even case. Let $\{1,\dots,n\}\smallsetminus \{a_1, \dots, a_k\}=\{b_1, \dots, b_k,b_{k+1}\}$ where $b_1<\dots<b_{k+1}$ and consider
\[
p=(x_{a_1}-x_{b_2})(x_{a_2}-x_{b_3})\dots (x_{a_k}-x_{b_{k+1}}).
\]
Proceed as above to show that $a_j<b_{j+1}$ by induction for all $1\le j\le k$ and conclude that the leading term of $p$ is equal to $x_{a_1}\dots x_{a_k}$. This completes the proof of \eqref{claim1} when $n$ is odd. 
\end{proof}

We end the paper with a discussion of the Macaulay inverse system of $G$, which is then used to give an easy proof that the Gorenstein ring $A$ has the Strong Lefschetz Property when $\ch \kk=0$.  

\subsection*{The Macaulay inverse system of $G$ and the SLP}
Let $\kk$ be a field, $Q=\kk[x_1, \dots, x_n]$, and let  $S=\kk[y_1, \dots, y_n]_{DP}$ denote the divided powers algebra in $n$ divided powers variables of degree $-1$ over $\kk$. We equip $S$ with a $Q$-module structure, called {\it contraction}, defined by extending the following action on monomials, by linearity in both arguments: 
\begin{equation*}
\label{eq:contraction}
x_1^{d_1}x_2^{d_2}\dots\, x_n^{d_n}\circ y_1^{(e_1)}y_2^{(e_2)}\dots\, y_n^{(e_n)} = \begin{cases}
y_1^{(e_1-d_1)}y_2^{(e_2-d_2)}\dots\, y_n^{(e_n-d_n)}, & \text{if } e_i-d_i\geq 0 \text{ for all } i\\
0, & \text{otherwise\,. }
\end{cases}
\end{equation*}
To simplify notation, we write $y_i$ instead of $y_i^{(1)}$. 

If $L$ is a homogeneous ideal of $Q$, then the {\it Macaulay inverse system} of $L$ is the graded $Q$-submodule of $S$ given by 
\begin{equation}
\label{eq:Iperp}
L^{-1}:=\{ g\in S \,\,\vert\,\, f\circ g=0 \text{ for all } f\in L\}. 
\end{equation}
It is well-known that $S$ is an injective hull of $\kk$ over $Q$. As a consequence, $L^{-1}$ is cyclic if and only if $Q/L$ is Gorenstein Artinian and 
$I=\Ann_Q(L^{-1})$; see for example \cite{IK}. 

An explicit generator of the inverse system of the Gorenstein ideal $G$ is known; we recall this in the following remark. 

\begin{remark}
\label{inverse}
By \cite[Lemma 2.7]{MRRG}, the Macaulay inverse system $G^{-1}$ of the  ideal $G$ is generated as a $Q$-module by the form  
\[
(x_1+\dots+x_n)^2\circ (y_1y_2\dots y_n)=\sum_{1\le i_1<\dots<i_{n-2}\le n} 2y_{i_1}y_{i_2}\dots y_{i_{n-2}}\,.
\]
\end{remark}

The inverse system of a Gorenstein ideal of $Q$ can be used to decide whether the corresponding quotient ring satisfies the SLP. We recall the relevant criterion below, after introducing Hessians, and then apply it to prove the SLP for the Gorenstein ring $A$. 

\begin{definition}
\label{hessian}
Let $d\ge 1$ be an integer and let $F$ be a homogeneous polynomial of degree $d$ in the divided powers algebra $S$ and set $A_F=Q/\Ann_Q(F)$. For an integer $i\ge 0$ we  define the $i$th Hessian matrix of $F$, denoted $\Hess^i(F)$, as follows. The rows and columns of the matrix are indexed by a monomial basis $B$ of $(A_F)_i$ and 
\[
\Hess^i(F)_{u,v}\coloneq (uv)\circ F\qquad \text{for $u,v\in B$},
\]
where we make the convention that, if $w\in A_F$, then $w\circ F\coloneq w'\circ F$, where $w'\in Q$ is the pre-image of $w$.
\end{definition} 

Next we recall a criterion for the Strong Lefschetz Property in \cite[Theorem 3.1]{Maeno-Watanabe}. 
\begin{theorem}[{\bf Hessian Criterion for SLP}, see \cite{Maeno-Watanabe}]
\label{hessian-criterion}
With the notation in \cref{hessian}, if $\ch\kk=0$ then an element $a_1x_1+\dots+a_nx_n\in (A_F)_1$ is a Strong Lefschetz element of $A_F$ if and only if 
\[
\det(\Hess^i(F))(a_1,\dots, a_n)\ne 0 \quad \text{ for all}\quad 0\le i\le \left\lfloor\frac{d}{2}\right\rfloor.
\]
\end{theorem}

It is known that the Gorenstein ring linked to an almost complete intersection generated by $n+1$ general quadrics has the SLP; see \cite[Corollary 2.7]{MM}. It is however not easily inferred from the proof of this result that our almost complete intersection ring $R$ is parametrized by an element in the open set on which the property holds. We provide a straightforward way to show the SLP for $A$ using the Hessian criterion below. 

\begin{proposition}
\label{A-has-SLP}
Let $n\ge 2$ be an integer and $Q=\kk[x_1,\dots,x_n]$, where $\kk$ is a field with $\ch\kk=0$. Then the Gorenstein ring  
%Adopt \cref{not:2} and set $\ell=\lfloor\frac{n-2}{2}\rfloor$.
\begin{equation*}
A = Q/\left( (x^2_1, \dots, x^2_n) :(x_1+\dots+ x_n)^2\right)
\end{equation*}
has the Strong Lefschetz Property. 
\end{proposition}

\begin{proof} 
Using \cref{inverse} and the fact that $\ch \kk=0$, we have $A=Q/\Ann_Q(F)$, where 
\[
F=\sum_{1\le i_1<\dots<i_{n-2}\le n} y_{i_1}y_{i_2}\dots y_{i_{n-2}}\,.
\]
To prove that $A$ has the SLP, we will apply \cref{hessian-criterion} for the element $x_1+x_2+\dots+x_n\in A_1$ with $d=n-2$. 

Setting $\ell=\left\lfloor\frac{n-2}{2}\right\rfloor$, we need to show that the Hessian matrix $\Hess^i(F)$ has full rank for all $0\le i\le \ell$, when evaluated at $y_1=1, \dots, y_n=1$. Fix $0\le i\le \ell$ and set 
\[
C=(\Hess^i(F))(1,1,\dots, 1).
\]
To compute $\Hess^i(F)$, observe that the basis $B$  of $A_i$ in \cref{hessian} consists of all square-free monomials in $Q_i$, as can be seen from \cref{Cataln-socle-gens}(2).  Let $u$ and $v$ be elements of $B$, and recall that the $(u,v)$ entry of $\Hess^i(F)$ is equal to $uv\circ F$. When $uv$ is not square-free, this entry is equal to $0$ . When $uv$ is square-free, $uv\circ F$ is equal to the sum of all square-free monomials of $S$ of degree $n-2-2i$ whose support is disjoint from the support of $uv$. Thus we have the following equality:
\[
C_{uv}=\left((uv)\circ F\right)(1,1,\dots, 1)=\begin{cases}
   0 &\text{if $\Supp(u)\cap \Supp(v)\ne\emptyset$} \\
   \binom{n-2i}{2} &\text{if $\Supp(u)\cap \Supp(v)=\emptyset$}\,.
\end{cases}
\] 
After factoring out $\binom{n-2i}{2}$, which is nonzero, the matrix $C$ is the binary matrix associated to the $i$-disjointness function, which is known to be invertible by \cite[Example 2.12]{comm-complexity}. 
\end{proof}

\section*{Acknowledgments} 
The project started at the meeting “Women in Commutative Algebra II” (WICA II) at CIRM Trento, Italy, October 16–20, 2023. The authors would like to thank the CIRM and the organizers for the invitation and financial support through the NSF grant DMS–2324929.
Support for attending the workshop also came from the AWM Travel Grants Program funded by NSF Grant DMS-2015440, Clay Mathematics Institute, and COS of Northeastern University.

This material was also partly  supported by the National Science Foundation under Grant No.\@ DMS-1928930 and by the Alfred P. Sloan Foundation under grant G-2021-16778, while the fifth author was in residence at the Simons Laufer Mathematical Sciences Institute (formerly MSRI) in Berkeley, California, during the Spring 2024 semester.  The first author was partially supported by the National Science Foundation Award No. 2418637. 

We would like to thank Mats Boij and Hailong Dao for conversations that took place at SLMath that helped shape the results and proofs in \cref{x1 reg}(1), respectively \cref{t:RAoverQ}.

\bibliographystyle{abbrv}
\bibliography{references}

\end{document}